\pdfoutput=1
\documentclass[11pt]{article}

\usepackage[T1]{fontenc}
\usepackage{lmodern}
\usepackage{microtype}
\usepackage[margin=1.18in]{geometry}
\usepackage{amsmath,amssymb,amsthm,mathtools}
\usepackage{xcolor}
\usepackage{graphicx}
\usepackage{hyperref}
\usepackage[nameinlink,capitalize,noabbrev]{cleveref}

\hypersetup{
  colorlinks=true,
  linkcolor=blue!45!black,
  citecolor=green!35!black,
  urlcolor=blue!50!black,
  pdftitle={Stochastic Domination of Gaussian Maxima: A Resolution of the Weak Simplex Conjecture},
  pdfauthor={Abhijeet Mulgund}
}

\allowdisplaybreaks

\newtheorem{theorem}{Theorem}[section]
\newtheorem{proposition}[theorem]{Proposition}
\newtheorem{lemma}[theorem]{Lemma}
\newtheorem{corollary}[theorem]{Corollary}
\newtheorem{problem}[theorem]{Open problem}
\newtheorem{resolvedproblem}[theorem]{Problem}
\theoremstyle{remark}
\newtheorem{remark}[theorem]{Remark}

\crefname{theorem}{theorem}{theorems}
\Crefname{theorem}{Theorem}{Theorems}
\crefname{proposition}{proposition}{propositions}
\Crefname{proposition}{Proposition}{Propositions}
\crefname{lemma}{lemma}{lemmas}
\Crefname{lemma}{Lemma}{Lemmas}
\crefname{corollary}{corollary}{corollaries}
\Crefname{corollary}{Corollary}{Corollaries}
\crefname{problem}{open problem}{open problems}
\Crefname{problem}{Open problem}{Open problems}
\crefname{resolvedproblem}{problem}{problems}
\Crefname{resolvedproblem}{Problem}{Problems}
\crefname{remark}{remark}{remarks}
\Crefname{remark}{Remark}{Remarks}

\newcommand{\R}{\mathbb R}
\newcommand{\E}{\mathbb E}
\newcommand{\Pp}{\mathbb P}
\newcommand{\one}{\mathbf 1}
\newcommand{\trans}{\mathsf T}
\newcommand{\ind}{\mathbf 1}
\newcommand{\ip}[2]{\langle #1,#2\rangle}
\newcommand{\norm}[1]{\lVert #1\rVert}
\newcommand{\cN}{\mathcal N}
\newcommand{\cF}{\mathcal F}
\newcommand{\cJ}{\mathcal J}

\newcommand{\M}{\mathsf M}
\DeclareMathOperator{\diag}{diag}
\DeclareMathOperator{\conv}{conv}
\DeclareMathOperator{\rank}{rank}

\title{Stochastic Domination of Gaussian Maxima: A Resolution of the Weak Simplex Conjecture}
\author{
  Abhijeet Mulgund \thanks{This work was supported by NSF under awards 2240532 and 2217023.} \\
  \small University of Illinois Chicago\\
  \small \href{mailto:mulgund2@uic.edu}{\texttt{mulgund2@uic.edu}}
}
\date{}

\begin{document}
\maketitle

\begin{abstract}
Let $R$ be an $m\times m$ correlation matrix satisfying
$R-\one\one^{\trans}/m\succeq0$, let $X\sim\cN(0,R)$, and let
$Z_1,\ldots,Z_m$ be independent standard Gaussian random variables.  We prove
\[
  \max_iX_i\leq_{\mathrm{st}}\max_iZ_i,
\]
with equality in distribution if and only if $R=I_m$.  We use this comparison
to resolve the Weak Simplex Conjecture: among $d+1$ equiprobable equal-energy
signals in $\R^d$ transmitted over an additive white Gaussian noise channel,
the regular simplex is the unique maximizer of the average probability of
correct maximum-likelihood decoding at every positive signal-to-noise ratio.  The same
comparison proves the Simplex Mean Width Conjecture and gives the exact
finite-energy performance of deterministic no-feedback AWGN codes with
equiprobable messages, no restriction on the number of channel uses, and a
maximal per-codeword energy constraint.

The proof uses a Gaussian product inequality for log-concave functions whose
first moments with respect to standard Gaussian measure vanish.  A variational
argument chooses one exponential tilt and one truncation endpoint in each
coordinate so that this product inequality applies and a Gaussian change of
measure returns all coordinates to the prescribed common threshold.  A strict
form of the product inequality also shows that, unless $R=I_m$,
$\Pp\{X\leq c\one\}>\Phi(c)^m$ for every finite $c$, and hence gives the
distributional equality statement.
A Lean formalization is available at
\url{https://github.com/abhmul/weak-simplex-conjecture-lean}.
\end{abstract}



\section{Introduction}

The Weak Simplex Conjecture traces back to an observation of Shannon reported
by Rice in 1950 \cite{Rice1950}; the name was later coined by Massey in 1988
\cite{Massey1988}.  It asks how to place $m=d+1$ equiprobable signals of equal
energy in $\R^d$ when the channel adds isotropic Gaussian noise and the receiver
uses maximum-likelihood decoding.  The conjecture asserts that the vertices of
a regular simplex maximize the average probability of correct decoding at
every positive signal-to-noise ratio.  If the unit signal vectors are
$x_1,\ldots,x_m$, if $G=(\ip{x_i}{x_j})_{i,j}$ is their Gram matrix, and if
$\lambda>0$ is the signal-to-noise ratio, then the correct decoding probability satisfies
\begin{equation}\label{eq:intro-ml}
  P_{\mathrm c}(x_1,\ldots,x_m;\lambda)
  =\frac{e^{-\lambda^2/2}}{m}
    \E\exp\!\left\{\lambda\max_{1\leq i\leq m}\xi_i\right\},
  \qquad \xi\sim\cN(0,G).
\end{equation}
We refer to \(\xi\) as the score vector for the code with Gram matrix \(G\),
since its coordinates are the random parts of the maximum-likelihood scores.
Thus the conjecture is equivalent to an extremal problem for the
moment-generating function of a Gaussian maximum.  The difficulty lies in the
failure of the usual entrywise covariance comparisons, such as Slepian's
inequality \cite{Slepian1962}, to settle the problem: an arbitrary competing
Gram matrix need not be entrywise comparable with the regular-simplex Gram
matrix.

Our main theorem is a stochastic comparison.  Write $J=\one\one^{\trans}$.  If
$R$ is a correlation matrix satisfying
\begin{equation}\label{eq:intro-psd-condition}
  R-\frac1mJ\succeq0,
\end{equation}
then, for $X\sim\cN(0,R)$ and independent standard Gaussian random variables
$Z_1,\ldots,Z_m$,
\begin{equation}\label{eq:intro-main}
  \Pp\!\left\{\max_iX_i\leq c\right\}
  \geq \Phi(c)^m
  =\Pp\!\left\{\max_iZ_i\leq c\right\}
  \qquad(c\in\R).
\end{equation}
Equivalently, $\max_iX_i\leq_{\mathrm{st}}\max_iZ_i$.  Furthermore, the two maxima have
the same distribution if and only if $R=I_m$.  The common threshold in
\eqref{eq:intro-main} is essential: the analogous comparison for arbitrary
threshold vectors is false; see \Cref{rem:unequal-thresholds}.  The direct
proof in \Cref{app:centered-product} gives the stronger conclusion that the
inequality is strict at every finite $c$ when $R\neq I_m$.

The coding problem is connected to \eqref{eq:intro-main} by adding the same
independent Gaussian random variable to every score and then rescaling.  On
covariance matrices, this operation is
\begin{equation}\label{eq:intro-normalization}
  R=\frac{m-1}{m}G+\frac1mJ.
\end{equation}
It sends the regular-simplex Gram matrix
$\frac{m}{m-1}(I_m-J/m)$ to $I_m$.  If
$\lambda=\sqrt{(m-1)/m}\,\mu$, then the corresponding Gaussian maxima satisfy
\begin{equation}\label{eq:intro-mgf-normalization}
 e^{-\mu^2/2}\E e^{\mu\max_iX_i}
 =e^{-\lambda^2/2}\E e^{\lambda\max_i\xi_i}.
\end{equation}
Combining \eqref{eq:intro-main}--\eqref{eq:intro-mgf-normalization} with
\eqref{eq:intro-ml} proves the Weak Simplex Conjecture.  The equality statement
also shows that the regular simplex is the unique maximizer, up to orthogonal
transformation and relabeling, at every $\lambda>0$.  This also yields
the exact optimal decoding probability in the deterministic no-feedback
finite-energy additive white-Gaussian noise (AWGN) model of \Cref{cor:finite-energy}.

Another consequence is the following bound on expected maxima of Gaussian random variables. Since the added Gaussian random variable has mean zero,
\begin{equation}\label{eq:intro-mean-width}
  \E\max_i\xi_i
  =\sqrt{\frac{m}{m-1}}\,\E\max_iX_i.
\end{equation}
Applying \eqref{eq:intro-main} to the identity function proves
\cite[Conjecture~1.4]{Litvak2018}, which is equivalent to the Simplex Mean
Width Conjecture \cite[Conjecture~1.1]{Litvak2018}.  Thus the regular simplex
uniquely maximizes mean width among convex hulls of $m$ points in the unit ball
of $\R^d$ whenever $d\geq m-1$.

\subsection{Proof idea}

The first ingredient is a Gaussian product inequality.  Let
$X\sim\cN(0,R)$, where $R$ is any correlation matrix.  If bounded
log-concave functions $f_i:\R\to[0,\infty)$ satisfy
\begin{equation}\label{eq:intro-first-moment}
  \int_{\R} zf_i(z)\,d\gamma_1(z)=0,
\end{equation}
then
\[
  \E\prod_i f_i(X_i)\geq\prod_i\int f_i\,d\gamma_1.
\]
After division by its Gaussian integral, each factor in
\eqref{eq:intro-first-moment} defines a probability measure of mean zero.  The
product inequality is a specialization of the centered Gaussian
forward--reverse Brascamp--Lieb theorem of Milman, Nakamura, and Tsuji
\cite{MilmanNakamuraTsuji2026}; we record the specialization and give a direct
proof in \Cref{app:centered-product}.

The indicator of $(-\infty,c]$ has the correct product but not the required
first moment.  We repair this by multiplying by an exponential and allowing
the endpoint to move.  Set
\begin{equation}\label{eq:intro-rH}
  r(s)=\frac{\varphi(s)}{\Phi(s)},
  \qquad H(s)=s+r(s).
\end{equation}
Then
\begin{equation}\label{eq:intro-factor}
  f_s(z)=e^{r(s)z}\ind_{(-\infty,H(s)]}(z)
\end{equation}
satisfies $\int zf_s(z)\,d\gamma_1(z)=0$.  Applying this exponential tilt
shifts the mean of $X$.  If one parameter $s_i$ is chosen
for each coordinate and $a_i=r(s_i)$, a Gaussian change of measure returns to
the event in \eqref{eq:intro-main} precisely when
\begin{equation}\label{eq:intro-compat}
  s+a-Ra=c\one.
\end{equation}
The logarithm of the Gaussian mass of \eqref{eq:intro-factor}, viewed as a
function of its endpoint $H(s)$, has derivative $r(s)$.  This identity turns
the coupled equations in \eqref{eq:intro-compat} into the stationarity
equations of a strictly concave finite-dimensional functional.  Its maximizer
also gives
\begin{equation}\label{eq:intro-scalar-target}
  \sum_{i=1}^m\log\Phi(s_i)
  +\frac12a^{\trans}(I_m-R)a
  \geq m\log\Phi(c),
\end{equation}
which is exactly the remaining estimate needed after the product inequality
and the Gaussian change of measure are combined.

The same product argument also gives the equality statement.  For the
truncated exponential factors above, the product inequality is strict whenever
$R\neq I_m$.  It follows that the common-threshold comparison is strict at
every finite $c$, so the maximum laws can agree only when $R=I_m$.  For
completeness, \Cref{app:maximum-rigidity} gives a different proof that applies
to arbitrary centered Gaussian vectors, without the unit-variance or
semidefinite hypotheses of the main theorem.

\subsection{Relation to prior work}\label{sec:prior-work}

The Weak Simplex Conjecture originates in Shannon's work on communication over Gaussian channels
\cite{Rice1950,Shannon1959}.  Balakrishnan proved several partial results
\cite{Balakrishnan1961,Balakrishnan1963,Balakrishnan1965}.  Landau and Slepian
gave an argument valid in dimension three \cite{LandauSlepian1966}; Tanner
explained why it does not extend to all dimensions
\cite{Tanner1970,Tanner1974}.  Later formulations appear in Farber's thesis,
Cover's problem list, and Massey's Shannon Lecture
\cite{Farber1968,Cover1987,Massey1988}.  The equal-codeword-energy Weak
Simplex Conjecture is distinct from the average-energy Strong Simplex
Conjecture, which Steiner disproved \cite{Steiner1994}.

The formulations in terms of Gaussian maxima and mean width were developed in
\cite{BalitskiyKarasevTsigler2017,KabluchkoLitvakZaporozhets2017,Litvak2018}.
Sun, Hu, and Lan proved the four-variable Gaussian maximum case and the
three-dimensional mean-width case \cite{SunHuLan2020}.  Gu\'edon and Souli
extend a Chevet comparison inequality to a broader class of smooth functions
and prove their comparison for a block-structured family of Gaussian vectors
\cite{GuedonSouli2026}.  Their Conjecture~1.6 records the Gaussian
expected-maximum formulation of the Simplex Mean Width Conjecture resolved
here.

The product inequality used in the proof is related to the
\v{S}id\'ak--Khatri inequality and the Gaussian correlation theorem
\cite{Khatri1967,Sidak1967,Harge2004,Royen2014,AssoulineChorSadovsky2024}.
Its present form follows from the forward--reverse Brascamp--Lieb framework of
Liu, Courtade, Cuff, and Verd{\'u} \cite{LiuCourtadeCuffVerdu2018}, together
with the centered Gaussian results of Milman, Nakamura, and Tsuji
\cite{Milman2025InverseBL,NakamuraTsuji2025Centered,NakamuraTsuji2026Saturation,MilmanNakamuraTsuji2026}.
Nakamura and Tsuji's equality theorem for two centered convex sets also implies
the equality case for finite symmetric rectangles.  The direct proof in
\Cref{app:centered-product} is shorter for the product form needed here.

\paragraph{Independent work on strictness and equality.}
After version~1 of this paper was posted, the present author and Su et al.\ independently
obtained the pointwise strictness and equality cases
\cite{SuYangXuI2026}.  A complete formal proof in the companion Lean
repository was committed on 19 July 2026\footnote{See
\href{https://github.com/abhmul/weak-simplex-conjecture-lean/commit/d7b84199cac385378abf365fa80872b67580ed93}
{See commit \texttt{d7b8419}}.}; the
branch was merged into the public \texttt{main} branch on 14 August 2026, and
the result was incorporated into the revision of this manuscript dated
20 August 2026.
Su et al.\ obtain the same strict common-threshold conclusion by strengthening
the first normalized self-convolution step: they first isolate a strict
bivariate rectangle gap and then use Royen's Gaussian correlation theorem to
group the remaining coordinate slabs.  The direct proof here instead proves
the full strict rectangle inequality by conditioning and induction. Thus the two pointwise
strictness proofs share the self-convolution framework of the direct
centered-product proof in \Cref{app:centered-product}, but establish the
decisive strict rectangle input differently.  We regard them as independent
parallel resolutions; the original question is retained as the resolved
\Cref{prob:equality-cases-resolved}.

\paragraph{Prior announced proofs.}
Two earlier preprints announce proofs: Goldsmith for the Simplex Mean Width
Conjecture and Pastore for the Weak Simplex Conjecture
\cite{Goldsmith2021,Pastore2023}.  Our proof is independent of both.  In
\Cref{app:prior-proofs} we give a version-specific account of unresolved steps
in the respective submissions arXiv:2112.03393v2 and arXiv:2306.13478v2 that prevent us from using their
announced conclusions.  That discussion concerns only those posted versions
and does not rule out repairs or later arguments.

\subsection{Organization}

\Cref{sec:results,sec:classification} state the main consequences and reduce
the coding problem to the covariance comparison.  The proof then proceeds in
three stages: \Cref{sec:centered-input} supplies the Gaussian product
inequality, \Cref{sec:alignment} constructs compatible centered factors, and
\Cref{sec:orthant-proof} combines them.  \Cref{sec:consequences} then returns
to the coding and geometric applications.  The appendices contain the
one-dimensional calculations, a direct proof and strict form of the product
inequality, a proof that a centered Gaussian vector whose maximum has the same
law as the maximum of independent standard Gaussians must have covariance
$I_m$, a version-specific discussion of two prior proof claims, and the details
needed to apply the Weak Simplex Conjecture to the AWGN sequence model of
\cite{PolyanskiyPoorVerdu2011}.

\section{Main results}\label{sec:results}

The central comparison is most naturally stated for the transformed
covariance $R$.  We state it first and then record the consequences obtained
after returning to the original score covariance $G$.

Throughout, $m\geq2$, $\one=(1,\ldots,1)^{\trans}\in\R^m$, and
$J=\one\one^{\trans}$.  We write $\gamma_k$ for standard Gaussian measure
on $\R^k$, and $\varphi$ and $\Phi$ for the standard Gaussian density and
distribution function, respectively.  All logarithms are natural, and vector inequalities
such as $x\leq y$ are coordinatewise.  A correlation matrix is a
positive-semidefinite matrix with unit diagonal and may be singular.  For real
random variables $U$ and $V$, we write $U\leq_{\mathrm{st}}V$ when
$\Pp\{U>t\}\leq\Pp\{V>t\}$ for every $t\in\R$.

Let
\[
  \M_X=\max_{1\leq i\leq m}X_i,
  \qquad
  \M_Z=\max_{1\leq i\leq m}Z_i,
\]
where $Z_1,\ldots,Z_m$ are independent standard Gaussian random variables.

\begin{theorem}\label{thm:orthant}
Let $R$ be an $m\times m$ correlation matrix satisfying
\[
  R-\frac1mJ\succeq0,
\]
and let $X\sim\cN(0,R)$.  Then
\begin{equation}\label{eq:stochastic-main}
  \M_X\leq_{\mathrm{st}}\M_Z.
\end{equation}
Equivalently,
\begin{equation}\label{eq:orthant-main}
  \Pp\{X\leq c\one\}\geq\Phi(c)^m
  \qquad\text{for every }c\in\R.
\end{equation}
Moreover,
\begin{equation}\label{eq:max-law-equality}
  \M_X\stackrel{d}{=}\M_Z
  \qquad\Longleftrightarrow\qquad
  R=I_m.
\end{equation}
\end{theorem}

The event in \eqref{eq:orthant-main} is the lower orthant
$(-\infty,c]^m$ with the \emph{same} threshold $c$ in every coordinate.  This
restriction cannot be removed.

\begin{remark}\label[remark]{rem:unequal-thresholds}
Let
\[
  R=
  \begin{pmatrix}
    1&-1/3&1/3\\
    -1/3&1&1/3\\
    1/3&1/3&1
  \end{pmatrix}.
\]
Then $R-J/3\succeq0$; indeed it is $2/3$ times the Gram matrix of
$e_1,-e_1,e_2$.  If $X\sim\cN(0,R)$, the correlation of $X_1$ and $X_2$ is
$-1/3$, so Sheppard's Lemma gives
\[
  \Pp\{X_1\leq0,X_2\leq0\}
  =\frac14-\frac{1}{2\pi}\arcsin\!\left(\frac13\right)
  <\frac14.
\]
As $L\to\infty$,
$\Pp\{X_1\leq0,X_2\leq0,X_3\leq L\}$ tends to the left-hand side,
whereas $\Phi(0)^2\Phi(L)$ tends to $1/4$.  Hence, for all sufficiently
large finite $L$,
\[
  \Pp\{X\leq(0,0,L)^{\trans}\}
  <\Phi(0)^2\Phi(L).
\]
Thus the comparison in \Cref{thm:orthant} can only apply to a
common-threshold statement; an arbitrary lower-orthant comparison does not
hold in general.
\end{remark}

\begin{remark}\label[remark]{rem:strict-common-threshold}
The applications below require only the equality-in-distribution statement in
\eqref{eq:max-law-equality}.  \Cref{thm:strict-orthant-refinement} proves the
stronger pointwise statement
\[
  R\neq I_m
  \quad\Longrightarrow\quad
  \Pp\{X\leq c\one\}>\Phi(c)^m
  \quad\text{for every finite }c.
\]
Thus equality at even one finite threshold \(c\) forces $R=I_m$.
The same pointwise refinement was obtained independently by Su et al.~\cite[Theorem~1]{SuYangXuI2026}.
An independent argument in \Cref{app:maximum-rigidity} proves the
equality-in-distribution statement for an arbitrary centered Gaussian vector.
\end{remark}

\subsection{Consequences for Gaussian maxima and simplex codes}

The advantage of the distributional comparison is that it controls every
nondecreasing function of the maximum, rather than one moment-generating
function at a time.  The equality statement also makes these comparisons
strict for strictly increasing integrable functions.
\begin{corollary}\label[corollary]{cor:monotone}
Under the hypotheses of \Cref{thm:orthant}, every nonnegative, nondecreasing
Borel function $\psi$ satisfies
\[
  \E\psi(\M_X)\leq\E\psi(\M_Z),
\]
including when one or both expectations are infinite.  If $R\neq I_m$,
$\psi:\R\to\R$ is strictly increasing, and $\psi(\M_X)$ and
$\psi(\M_Z)$ are integrable, then
\[
  \E\psi(\M_X)<\E\psi(\M_Z).
\]
In particular, for every $\lambda>0$,
\[
  \E e^{\lambda\M_X}<\E e^{\lambda\M_Z},
  \qquad
  \E\M_X<\E\M_Z
  \qquad(R\neq I_m).
\]
\end{corollary}

We now return from the transformed covariance $R$ to the score covariance $G$.
Set
\begin{equation}\label{eq:G-delta}
  G_\triangle=\frac{m}{m-1}\left(I_m-\frac1mJ\right).
\end{equation}
This is the Gram matrix of the vertices of a regular $(m-1)$-simplex
inscribed in the unit sphere.  Conversely, any unit-vector configuration with
this Gram matrix is a regular simplex, up to an orthogonal transformation and
relabeling.  Under the covariance transformation in \eqref{eq:intro-normalization},
$G_\triangle$ is sent to $I_m$.  Thus \Cref{thm:orthant} compares every score
covariance with the regular-simplex covariance.

\begin{corollary}\label[corollary]{cor:mgf}
Let $G$ be any $m\times m$ correlation matrix, let
$\xi\sim\cN(0,G)$, and let
$\xi^\triangle\sim\cN(0,G_\triangle)$.  Then
\begin{equation}\label{eq:mgf-comparison}
  \E\exp\!\left\{\lambda\max_i\xi_i\right\}
  \leq
  \E\exp\!\left\{\lambda\max_i\xi_i^\triangle\right\}
  \qquad(\lambda>0).
\end{equation}
If $G\neq G_\triangle$, then \eqref{eq:mgf-comparison} is strict for every
$\lambda>0$.  Equivalently, equality for any one $\lambda>0$ holds if and
only if $G=G_\triangle$.
\end{corollary}

Combining the preceding MGF comparison with the maximum-likelihood identity
in \Cref{sec:classification} gives the coding theorem in its classical form.

\begin{corollary}[Weak Simplex Conjecture]\label[corollary]{cor:wsc}
Let $m=d+1$, let $x_1,\ldots,x_m$ be unit vectors in $\R^d$, and let the
message $I$ be uniform on $\{1,\ldots,m\}$.  Under the channel
\[
  Y=\lambda x_I+Z,
  \qquad Z\sim\cN(0,I_d),\quad \lambda>0,
\]
the average probability of correct maximum-likelihood decoding is no larger
than for a regular simplex.  For every fixed $\lambda>0$, equality holds if
and only if the $x_i$ form a regular simplex.  Thus the regular simplex is the
unique maximizer up to rotation and relabeling.
\end{corollary}

The single-parameter MGF rigidity in \Cref{cor:mgf} and the strict coding
equality case in \Cref{cor:wsc} were obtained independently by Su et al.~\cite[Theorems~2 and~3(a)]{SuYangXuI2026}.

The comparison in \Cref{cor:mgf} applies to Gram matrices of $m$ unit vectors
in any Euclidean space.  When the ambient dimension is at least $m-1$, the
regular simplex is attainable.  In lower dimension it remains a universal
upper bound, but we do not identify the best Gram matrix under a rank constraint.

The expected-maximum comparison extends from correlation matrices to
covariance matrices whose marginal variances are at most one.  This is the
form that directly contains the geometric problem.

\begin{corollary}[Gaussian maxima and the Simplex Mean Width Conjecture]
\label[corollary]{cor:mean-width}
Let $S$ be an $m\times m$ positive-semidefinite matrix with $S_{ii}\leq1$,
and let $Y\sim\cN(0,S)$.  Then
\begin{equation}\label{eq:expected-max}
  \E\max_iY_i
  \leq \E\max_i\xi_i^\triangle
  =\sqrt{\frac{m}{m-1}}\,\E\max_i Z_i,
\end{equation}
where $\xi^\triangle\sim\cN(0,G_\triangle)$.  Equality holds if and only if
$S=G_\triangle$.

Consequently, among convex hulls of $m$ points contained in the Euclidean unit
ball of $\R^d$, for every $d\geq m-1$, the regular $(m-1)$-simplex uniquely
maximizes mean width, up to rotation and relabeling.
\end{corollary}

The equality case in \Cref{cor:mean-width} was obtained independently by
Su et al.~\cite[Corollary~4]{SuYangXuI2026}.

\subsection{Finite-energy AWGN consequence}

The simplex theorem also determines the optimal probability of error for the
AWGN channel when the number of channel uses is unrestricted.  Following
\cite[Definition~1]{PolyanskiyPoorVerdu2011}, fix $E\geq0$ and $N_0>0$ and
consider the AWGN sequence model $Y=X+Z$ over $\R^\infty$, with independent
noise coordinates of variance $N_0/2$.  An $(E,M,\epsilon)$ code consists of
$M$ equiprobable codewords $c_i\in\ell_2$ satisfying
$\norm{c_i}^2\leq E$, together with a decoder whose average error probability
is at most $\epsilon$.

This model imposes no constraint on the number of channel uses or nonzero
codeword coordinates.  For each fixed codebook, equal-prior maximum-likelihood
decoding is Bayes-optimal.  Let $P_{\mathrm c}^*(E,M)$ be the largest average
correct-decoding probability over admissible encoders and decoders, and let
$M^*(E,\epsilon)$ be the largest number of messages that can be communicated
with error at most $\epsilon$.  Let $E^*(M,\epsilon)$ be the minimum energy
needed to communicate $M$ messages with error at most $\epsilon$.  For
$M\geq2$, define
\begin{equation}\label{eq:pm-def}
  p_M(a)=\E\Phi(W+a)^{M-1},\qquad W\sim\cN(0,1).
\end{equation}

\begin{corollary}\label[corollary]{cor:finite-energy}
In the preceding model, the optimal correct-decoding probability is
\begin{equation}\label{eq:pc-opt}
  P_{\mathrm c}^*(E,M)
  =p_M\!\left(\sqrt{\frac{2EM}{(M-1)N_0}}\right),
  \qquad M\geq2.
\end{equation}
Hence, for $0\leq\epsilon<1$,
\begin{equation}\label{eq:M-star}
  M^*(E,\epsilon)
  =\max\!\left(
    \{1\}\cup
    \left\{M\geq2:
      p_M\!\left(\sqrt{\frac{2EM}{(M-1)N_0}}\right)
      \geq1-\epsilon
    \right\}
  \right).
\end{equation}
If $0<\epsilon<1-1/M$, there is a unique $a_{M,\epsilon}>0$ satisfying
$p_M(a_{M,\epsilon})=1-\epsilon$, and the minimum energy for $M$ messages at error \(\epsilon\) is
\begin{equation}\label{eq:E-min}
  E^{*}(M,\epsilon)
  =\frac{N_0}{2}\frac{M-1}{M}a_{M,\epsilon}^2.
\end{equation}
For $\epsilon\geq1-1/M$, $E^{*}(M,\epsilon)$ is zero, whereas for
$\epsilon=0$ it is infinite.
\end{corollary}

\paragraph{Equality in the finite-energy bound.}
For $E>0$, Su et al.\ independently characterize all optimizers
\cite[Theorem~3(b)]{SuYangXuI2026}: an admissible $M$-message codebook attains
\eqref{eq:pc-opt} if and only if
\[
  \sum_{i=1}^M c_i=0,\qquad
  \norm{c_i}^2=E,\qquad
  \ip{c_i}{c_j}=-\frac{E}{M-1}\quad(i\neq j).
\]
Thus the optimizer is a centered regular simplex of circumradius $\sqrt E$,
up to isometry of its span and relabeling, and every optimal codeword exhausts
the energy allowance.  This equality characterization supplements the
optimal-value formula above.

\begin{remark}[Relation to pulse-position modulation]
Orthogonal pulse-position modulation (PPM) has codewords
$\sqrt E e_1,\ldots,\sqrt E e_M$. After subtracting its common center, it becomes
a regular simplex of squared radius $E(M-1)/M$.  Its correct-decoding
probability is therefore
\[
  p_M\!\left(\sqrt{\frac{2E}{N_0}}\right).
\]
Equivalently, the full-energy simplex at energy $E$ has the same performance
as orthogonal PPM at energy $EM/(M-1)$.  \Cref{cor:wsc} then explains why the 
orthogonal PPM construction in \cite{PolyanskiyPoorVerdu2011} nearly obtains \(E^{*}(M, \epsilon)\).
\end{remark}

We now turn to the proof of \Cref{thm:orthant}, beginning with the
passage from the Gram matrix to the transformed covariance $R\succeq J/m$.

\section{From Gaussian classification to a covariance comparison}
\label{sec:classification}

We now establish the reduction transforming the original coding
problem with signal vectors $x_1,\ldots,x_m$ to \Cref{thm:orthant}.

\subsection{The maximum-likelihood integral}

Let $x_1,\ldots,x_m$ be unit vectors in $\R^d$.  Conditional on message $i$,
the channel output \(Y\) has density
\[
  p_i(y)=(2\pi)^{-d/2}
  \exp\!\left(-\frac12\|y-\lambda x_i\|^2\right).
\]
Since $\|x_i\|=1$,
\[
  p_i(y)
  =(2\pi)^{-d/2}e^{-\|y\|^2/2-\lambda^2/2}
    e^{\lambda\ip{y}{x_i}}.
\]
Only the last factor depends on $i$, so maximum-likelihood decoding selects an
index maximizing $\ip{y}{x_i}$.  Let $D_i \subseteq \R^d$ be the decision region assigned to
message $i$, using any measurable rule on ties.  The regions form a partition,
and $p_i(y)=\max_jp_j(y)$ on $D_i$.  Hence
\begin{align*}
  P_{\mathrm c}(x_1,\ldots,x_m;\lambda)
  &=\frac1m\sum_{i=1}^m\int_{D_i}p_i(y)\,dy\\
  &=\frac{e^{-\lambda^2/2}}{m}
    \int_{\R^d}e^{\lambda\max_i\ip{y}{x_i}}\,d\gamma_d(y).
\end{align*}
If $Z\sim\cN(0,I_d)$ and $\xi_i=\ip{Z}{x_i}$, then
$\xi\sim\cN(0,G)$ for $G=(\ip{x_i}{x_j})$, and we get
\begin{equation}\label{eq:ml-integral}
  P_{\mathrm c}(x_1,\ldots,x_m;\lambda)
  =\frac{e^{-\lambda^2/2}}{m}
    \E\exp\!\left\{\lambda\max_i\xi_i\right\}.
\end{equation}
This also shows directly that the average success probability is independent
of the rule used on tie sets.  The code has now been reduced to its score
covariance $G$.  We next change that covariance so that the regular simplex is
represented by independent coordinates.

\subsection{Adding a common Gaussian random variable}

The variance of the added variable is determined by the simplex.  For the
regular-simplex Gram matrix, every off-diagonal entry equals $-1/(m-1)$.
Adding a common Gaussian variable of variance $1/(m-1)$ therefore cancels those
off-diagonal covariances, and a final rescaling restores unit marginal
variances.  We apply the same operation to every other Gram matrix competing with the simplex.

Set
\begin{equation}\label{eq:alpha}
  \alpha=\frac{m-1}{m}.
\end{equation}
For a correlation matrix $G$, define
\begin{equation}\label{eq:R-from-G}
  R=\alpha G+(1-\alpha)J
   =\frac{m-1}{m}G+\frac1mJ.
\end{equation}
Then $R$ is a correlation matrix and $R-J/m=\alpha G\succeq0$.
Conversely, if a correlation matrix $R$ satisfies $R-J/m\succeq0$, then
\begin{equation}\label{eq:G-from-R}
  G=\alpha^{-1}\left(R-\frac1mJ\right)
\end{equation}
is positive semidefinite with unit diagonal.  Thus the hypothesis in
\Cref{thm:orthant} is exactly the image of the class of score covariance
matrices. Furthermore \(G_{\triangle}\) is transformed to \(I_m\) under this mapping.

For a probabilistic interpretation of this transform let $\xi\sim\cN(0,G)$ and let
$B\sim\cN(0,1/(m-1))$ be independent.  Define
\begin{equation}\label{eq:X-common}
  X_i=\sqrt\alpha\,(\xi_i+B),
  \qquad 1\leq i\leq m.
\end{equation}
Then $X\sim\cN(0,R)$ since 
\[
  \operatorname{Cov}(X)
  =\alpha\left(G+\frac1{m-1}J\right)
  =R.
\]
Because the same variable $B$ is added to every coordinate,
\[
  \max_iX_i=\sqrt\alpha\left(\max_i\xi_i+B\right).
\]
With $\lambda=\sqrt\alpha\,\mu$, we therefore obtain
\begin{align}
 e^{-\mu^2/2}\E e^{\mu\max_iX_i}
 &=e^{-\mu^2/2}\E e^{\lambda\max_i\xi_i}\E e^{\lambda B}\notag\\
 &=e^{-\lambda^2/2}\E e^{\lambda\max_i\xi_i}.
 \label{eq:normalized-mgf}
\end{align}

\begin{remark}[Role of the semidefinite condition]
\label[remark]{rem:semidefinite-condition}
The condition $R\succeq J/m$ is equivalent to
\begin{equation}\label{eq:psd-quadratic}
  u^{\trans}Ru
  \geq\frac1m\left(\sum_{i=1}^m u_i\right)^2
  \qquad(u\in\R^m).
\end{equation}
It does not impose an entrywise sign pattern: matrices in this class may have
both positive and negative correlations, as
\Cref{rem:unequal-thresholds} shows.  It records exactly that $R$ is the image
of a score covariance under \eqref{eq:R-from-G}.  The transformation may also
change rank; for example, $G_\triangle$ has rank $m-1$, whereas its image
$I_m$ has full rank.  Whenever the proof uses a Gram representation of the
matrix in \eqref{eq:G-from-R}, we choose unit vectors in
$\R^{\rank G}$, so no ambient dimension must be fixed in advance.
\end{remark}

The reduction is complete: by \eqref{eq:ml-integral} and
\eqref{eq:normalized-mgf}, the Weak Simplex Conjecture (\Cref{cor:wsc}) follows once
\Cref{thm:orthant} is proved for the covariance class $R\succeq J/m$.  The
next section supplies the Gaussian product inequality used in that proof.

\section{A Gaussian product inequality}\label{sec:centered-input}

The event $\{X\leq c\one\}$ is a product of one-dimensional indicators
$\ind_{\{X_i\leq c\}}$, so a Gaussian product inequality is a natural
starting point.  The key result we use is a Gaussian product inequality that applies to bounded log-concave factors that
are centered with respect to the standard Gaussian measure.  This centering
hypothesis will be the only obstruction to using the indicators directly, and \Cref{sec:alignment} will address how to resolve that obstruction.

A function $f:\R\to[0,\infty)$ is log-concave if
\[
  f((1-\theta)x+\theta y)
  \geq f(x)^{1-\theta}f(y)^\theta
  \qquad(x,y\in\R,\ 0\leq\theta\leq1).
\]
We allow $f$ to vanish; in particular, indicators of intervals and the
truncated exponentials used below are log-concave.

\begin{lemma}
\label[lemma]{thm:centered-product}
Let $R$ be an $m\times m$ correlation matrix and let $X\sim\cN(0,R)$.  For
each $i$, let $f_i:\R\to[0,\infty)$ be bounded and log-concave, with positive
Gaussian integral, and suppose that
\begin{equation}\label{eq:zero-first-moment}
  \int_{\R}z f_i(z)\,d\gamma_1(z)=0.
\end{equation}
Then
\begin{equation}\label{eq:centered-product}
  \E\prod_{i=1}^m f_i(X_i)
  \geq \prod_{i=1}^m \E f_i(X_i)
  =\prod_{i=1}^m\int_{\R}f_i\,d\gamma_1.
\end{equation}
The covariance matrix $R$ may be singular.
\end{lemma}

To interpret \eqref{eq:zero-first-moment}, divide $f_i$ by its Gaussian
integral.  The measure
\[
  \frac{f_i\,d\gamma_1}{\int f_i\,d\gamma_1}
\]
is then a probability measure, and \eqref{eq:zero-first-moment} says precisely
that this measure has mean zero.  The lemma therefore compares the correlated
product with the independent product under a coordinatewise centering
condition.

The product inequality itself does not use the condition $R\succeq J/m$ from
\Cref{thm:orthant}.  That condition enters only in \Cref{sec:alignment}, where
we need it to supply a Gram representation of
$\alpha^{-1}(R-J/m)$.

\Cref{thm:centered-product} is a specialization of the Gaussian
forward--reverse Brascamp--Lieb theorem of Milman, Nakamura, and Tsuji
\cite[Theorem~1.10]{MilmanNakamuraTsuji2026}.  The exact parameter assignment
is given in \Cref{app:frbl-specialization}.  For self-containment,
\Cref{app:centered-product} also gives a direct proof of precisely this
one-dimensional-factor specialization, using the Gaussian correlation inequality, normalized self-convolution, and the central limit
theorem.  The direct argument additionally supplies the strict form needed in
\Cref{rem:strict-common-threshold}.  Su et al.\ independently obtain the same
strict product conclusion by strengthening the first self-convolution step
with a Royen-based strict rectangle inequality
\cite[Sections~V-A--V-B]{SuYangXuI2026}.

For the raw indicators, the two sides of the product inequality are already
the two sides of \eqref{eq:orthant-main}.  Indeed, for
$f_i(z)=\ind_{(-\infty,c]}(z)$,
\[
  \E\prod_{i=1}^m f_i(X_i)=\Pp\{X\leq c\one\},
  \qquad
  \prod_{i=1}^m\int f_i\,d\gamma_1=\Phi(c)^m.
\]
However, the required first moment is
\[
  \int zf_i(z)\,d\gamma_1(z)=-\varphi(c)\neq0.
\]
Multiplying by an exponential can restore the missing first moment, but the
same weight shifts the mean of the multivariate Gaussian vector $X$.  The
inequality \eqref{eq:centered-product} is therefore available only after the
centering condition \eqref{eq:zero-first-moment} has been reconciled with the
target event $\{X\leq c\one\}$.  That compatibility problem is the subject
of the next section.

\section{Adaptive threshold alignment}\label{sec:alignment}

Fix a target threshold $c\in\R$.  The raw indicators of the event
$\{X\leq c\one\}$ fail the centering hypothesis in
\Cref{thm:centered-product}.  We therefore replace the $i$th indicator by a
truncated exponential
\[
  f_i(z)=e^{a_i z}\ind_{(-\infty,b_i]}(z).
\]
The parameter $a_i$ must center the one-dimensional factor, while the endpoint
$b_i$ must compensate for the joint Gaussian mean displacement $Ra$ created by
all the tilts.  So our construction will have a one-dimensional part and a coupled
part.  We first derive the two required identities, then solve the
one-dimensional centering equation, and finally choose all endpoints by a
strictly concave variational problem.

\subsection{Required conditions for the tilted factors}

For vectors $a\in(0,\infty)^m$ and $b\in\R^m$, consider
\begin{equation}\label{eq:tilted-factors}
  f_i(z)=e^{a_i z}\ind_{(-\infty,b_i]}(z).
\end{equation}
The Gaussian change-of-measure identity below describes the multivariate effect
of these parameters.  We include the short proof that also covers the singular covariance case.

\begin{lemma}\label[lemma]{lem:gaussian-shift}
Let $X\sim\cN(0,R)$ for a positive-semidefinite matrix $R$.  For every
$a\in\R^m$ and every nonnegative Borel function
$F:\R^m\to[0,\infty]$,
\begin{equation}\label{eq:gaussian-shift}
  \E\!\left[e^{a^{\trans}X}F(X)\right]
  =e^{\frac12a^{\trans}Ra}\E F(X+Ra).
\end{equation}
\end{lemma}

\begin{proof}
Choose a matrix $T$ and a standard Gaussian vector $Z\sim\cN(0,I_d)$ such
that $X=TZ$ and $TT^{\trans}=R$.  With $\theta=T^{\trans}a$, let
$\varphi_d$ denote the standard Gaussian density on $\R^d$.  Completing the square gives
\[
  e^{\theta^{\trans}z-\|\theta\|^2/2}\varphi_d(z)=\varphi_d(z-\theta).
\]
A change of variables therefore yields
\begin{align*}
 \E\!\left[e^{a^{\trans}X-\frac12a^{\trans}Ra}F(X)\right]
 &=\E\!\left[e^{\theta^{\trans}Z-\|\theta\|^2/2}F(TZ)\right]\\
 &=\E F(T(Z+\theta))
 =\E F(X+Ra),
\end{align*}
which is equivalent to \eqref{eq:gaussian-shift}.
\end{proof}

Two identities are required.  First, each coordinate must satisfy the centering
hypothesis of \Cref{thm:centered-product}:
\begin{equation}\label{eq:factor-first-moment}
  \int_{\R}z e^{a_i z}\ind_{(-\infty,b_i]}(z)\,d\gamma_1(z)=0
  \qquad(1\leq i\leq m).
\end{equation}
Second, after the Gaussian change of measure, the weighted event must become
the target event $\{X\leq c\one\}$.  Applying \Cref{lem:gaussian-shift} with
$F(x)=\ind_{\{x\leq b\}}$ gives
\[
  \E\!\left[e^{a^{\trans}X}\ind_{\{X\leq b\}}\right]
  =e^{\frac12a^{\trans}Ra}\Pp\{X\leq b-Ra\}.
\]
We must therefore choose the endpoints so that
\begin{equation}\label{eq:threshold-compatibility}
  b-Ra=c\one.
\end{equation}
Under this condition,
\begin{equation}\label{eq:event-shift}
  \E\!\left[e^{a^{\trans}X}\ind_{\{X\leq b\}}\right]
  =e^{\frac12a^{\trans}Ra}\Pp\{X\leq c\one\}.
\end{equation}

Completing the square also gives the Gaussian integral of an individual
factor:
\[
  \int_{-\infty}^{b_i}e^{a_i z}\,d\gamma_1(z)
  =e^{a_i^2/2}\Phi(b_i-a_i).
\]
Consequently, once \eqref{eq:factor-first-moment} and
\eqref{eq:threshold-compatibility} hold,
\Cref{thm:centered-product} and \eqref{eq:event-shift} give
\begin{equation}\label{eq:tilt-reduction}
  \Pp\{X\leq c\one\}
  \geq
  \exp\!\left(\frac12a^{\trans}(I_m-R)a\right)
  \prod_{i=1}^m\Phi(b_i-a_i).
\end{equation}
It remains to choose $a$ and $b$ so that both required identities hold and
\begin{equation}\label{eq:tilt-reduction-target}
  \exp\!\left(\frac12a^{\trans}(I_m-R)a\right)
  \prod_{i=1}^m\Phi(b_i-a_i)
  \geq \Phi(c)^m.
\end{equation}
\Cref{sec:one-d-calc} solves the one-dimensional centering condition and
identifies the potential used to couple the coordinates.
\Cref{sec:variational} then enforces threshold compatibility and proves
\eqref{eq:tilt-reduction-target} simultaneously.

\subsection{The one-dimensional centered factors}\label{sec:one-d-calc}
The first condition in \eqref{eq:factor-first-moment} is one-dimensional and
can be solved explicitly.  For $a,b\in\R$, let
$f_{a,b}(z)=e^{az}\ind_{(-\infty,b]}(z)$ and set
\[
  s=b-a.
\]
\(s\) will allow us to parameterize the family of \(f_{a,b}\) satisfying \eqref{eq:factor-first-moment}.

\begin{lemma}\label[lemma]{lem:local-centering}
For $a,b,s$ as above,
\begin{align}
 \int f_{a,b}\,d\gamma_1
 &=e^{a^2/2}\Phi(s),\label{eq:local-mass}\\
 \int zf_{a,b}(z)\,d\gamma_1(z)
 &=e^{a^2/2}\bigl(a\Phi(s)-\varphi(s)\bigr).
 \label{eq:local-first-moment}
\end{align}
Consequently, the first moment vanishes if and only if
\begin{equation}\label{eq:local-centering-curve}
  a=r(s):=\frac{\varphi(s)}{\Phi(s)},
  \qquad
  b=H(s):=s+r(s).
\end{equation}
\end{lemma}

\begin{proof}
The identity
$e^{az}\varphi(z)=e^{a^2/2}\varphi(z-a)$ gives
\eqref{eq:local-mass}.  Differentiating that identity with respect to $a$ while
keeping $b$ fixed gives \eqref{eq:local-first-moment}.  Since $\Phi(s)>0$, the
last assertion follows by solving $a\Phi(s)-\varphi(s)=0$.
\end{proof}

Thus a single parameter $s$ determines both the tilt and the endpoint.
For $Z\sim\cN(0,1)$,
\[
  r(s)=-\E[Z\mid Z\leq s],
  \qquad
  H(s)=\E[s-Z\mid Z\leq s],
\]
so both quantities are positive.

The coupled construction is most conveniently written in terms of the endpoint
$b=H(s)$.  We therefore introduce a potential whose derivative with respect to
$b$ is the required tilt.  Along the centered family, define
\begin{equation}\label{eq:F-def}
  \cF(H(s))
  =\log\!\int f_{r(s),H(s)}\,d\gamma_1
  =\log\Phi(s)+\frac12r(s)^2.
\end{equation}
The map $H$ is a bijection from $\R$ onto $(0,\infty)$, so this definition
determines $\cF$ on the entire endpoint domain.  Differentiating
\eqref{eq:F-def} with respect to $s$ gives
\[
  \frac{d}{ds}\cF(H(s))
  =r(s)+r(s)r'(s)
  =r(s)H'(s),
\]
and hence
\[
  \cF'(H(s))=r(s).
\]
Thus the derivative of the logarithmic mass of the centered factor is exactly
its exponential tilt.  The remaining properties are collected below and
proved in \Cref{app:truncated-functions}; \Cref{fig:tilt-functions} displays
the three functions separately.

\begin{proposition}\label[proposition]{prop:tilt-functions}
The function $r$ is a strictly decreasing bijection from $\R$ onto
$(0,\infty)$, and $H$ is a strictly increasing bijection from $\R$ onto
$(0,\infty)$.  The function $\cF:(0,\infty)\to(-\infty,0)$ defined by
\eqref{eq:F-def} is a strictly increasing, strictly concave bijection.  Moreover,
\begin{equation}\label{eq:F-properties}
  \cF'(H(s))=r(s)
  \qquad(s\in\R).
\end{equation}
\end{proposition}

Because $r$ and $H$ are bijective, the centered family
$(s,a,b)=(s,r(s),H(s))$ can be parametrized by $s\in\R$, by the tilt $a>0$,
or by the endpoint $b>0$.  The endpoint parametrization is useful because
\eqref{eq:threshold-compatibility} is an equation for $b$.

\begin{figure}[htbp]
  \centering
  \begin{minipage}[c]{0.315\linewidth}
    \centering
    \includegraphics[width=\linewidth]{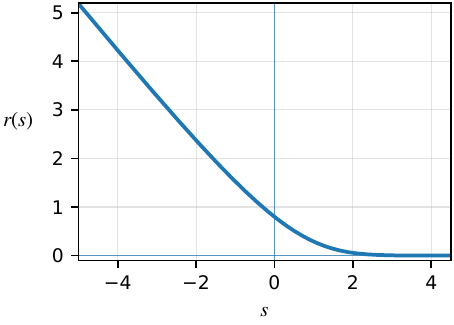}
  \end{minipage}\hfill
  \begin{minipage}[c]{0.315\linewidth}
    \centering
    \includegraphics[width=\linewidth]{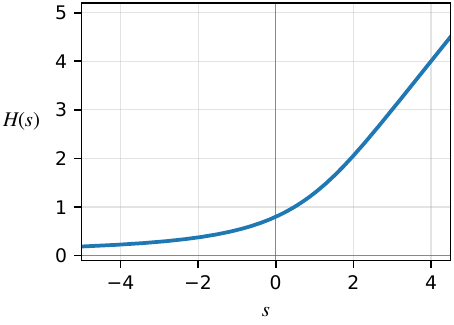}
  \end{minipage}\hfill
  \begin{minipage}[c]{0.315\linewidth}
    \centering
    \includegraphics[width=\linewidth]{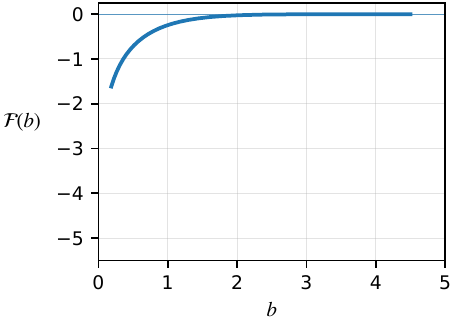}
  \end{minipage}
  \caption{The centering tilt $r(s)$, the corresponding endpoint $H(s)$, and
  the endpoint potential $\cF(b)$.  As $s$ increases from $-\infty$ to
  $+\infty$, $r(s)$ decreases from $+\infty$ to $0$, while $H(s)$ increases
  from $0$ to $+\infty$.  The potential is increasing and strictly concave on
  $(0,\infty)$.}
  \label{fig:tilt-functions}
\end{figure}

\subsection{Deriving the variational problem}\label{sec:variational}

Apply $r$ and $H$ coordinatewise and write
$a_i=r(s_i)$ and $b_i=H(s_i)$.  Then the factors in
\eqref{eq:tilted-factors} satisfy the centering condition for every
$s\in\R^m$.  The remaining problem is to impose the coupled identity
$b-Ra=c\one$.  We now derive a concave functional whose stationarity equations
are exactly this compatibility condition.

For an endpoint $b_i=H(s_i)$, the corresponding tilt is
$a_i=r(s_i)=\cF'(b_i)$.  Thus \eqref{eq:threshold-compatibility} is the
fixed-point equation
\begin{equation}\label{eq:fixed-point-b}
  b=c\one+Ra=c\one+R\cF'(b),
\end{equation}
where $\cF'$ is applied coordinatewise and $b\in(0,\infty)^m$.

The semidefinite condition $R \succeq J/m$ now enters. Recall
$\alpha=(m-1)/m$ and set
\[
  G=\alpha^{-1}\left(R-\frac1mJ\right).
\]
Then $G$ is a correlation matrix.  Choose unit vectors
$x_1,\ldots,x_m\in\R^d$, with $d=\rank G$, such that
$G_{ij}=\ip{x_i}{x_j}$.  For any $a\in\R^m$, let
$\bar a=m^{-1}\sum_i a_i$.  Since $R=\alpha G+J/m$,
\begin{equation}\label{eq:Ra-decomposition}
  (Ra)_i
  =\alpha\ip{x_i}{\sum_{j=1}^m a_jx_j}+\bar a.
\end{equation}
This decomposition suggests the variables in which to solve
\eqref{eq:fixed-point-b}. Put
\[
  v=\sum_{j=1}^m a_jx_j,
  \qquad
  q=c+\bar a.
\]
Then \eqref{eq:fixed-point-b} is equivalent to
\begin{equation}\label{eq:qv-target}
  b_i=q+\alpha\ip{x_i}{v},
  \qquad
  q-c=\frac1m\sum_{i=1}^m a_i,
  \qquad
  v=\sum_{i=1}^m a_i x_i.
\end{equation}
After substituting $a_i=\cF'(b_i)$ and
$b_i=q+\alpha\ip{x_i}{v}$, the last two equations in
\eqref{eq:qv-target} are the stationarity equations of the following
functional.  For $(q,v)\in\R\times\R^d$, define
\begin{equation}\label{eq:Jc-def}
 \cJ_c(q,v)
 =\sum_{i=1}^m\cF\!\left(q+\alpha\ip{x_i}{v}\right)
  -\frac\alpha2\|v\|^2
  -\frac m2(q-c)^2
\end{equation}
on the open convex set
\begin{equation}\label{eq:Jc-domain}
  \mathcal D
  =\left\{(q,v):q+\alpha\ip{x_i}{v}>0
    \text{ for every }i\right\}.
\end{equation}
The domain is nonempty because $(q,0)\in\mathcal D$ whenever $q>0$.  If
$b_i=q+\alpha\ip{x_i}{v}$ and $a_i=\cF'(b_i)$, then
\begin{align}\label{eq:Jc-gradients}
  \partial_q\cJ_c(q,v)
  &=\sum_{i=1}^m a_i-m(q-c),\\
  \nabla_v\cJ_c(q,v)
  &=\alpha\left(\sum_{i=1}^m a_i x_i-v\right).
\end{align}
Setting these gradients to zero gives the last two equations in
\eqref{eq:qv-target}.  Therefore, with
$b_i=q+\alpha\ip{x_i}{v}$ and $a_i=\cF'(b_i)$, every stationary point of
$\cJ_c$ satisfies $b-Ra=c\one$.

The value of the same functional supplies the probability estimate.  At a
stationary point,
\[
  \frac\alpha2\|v\|^2+\frac m2(q-c)^2
  =\frac12a^{\trans}Ra,
\]
while \eqref{eq:F-def} gives
$\cF(b_i)=\log\Phi(s_i)+a_i^2/2$.  Hence $\cJ_c(q,v)$ is precisely the
logarithm of the lower bound in \eqref{eq:tilt-reduction}.  The symmetric trial
point $(H(c),0)$ satisfies
\[
  \cJ_c(H(c),0)
  =m\cF(H(c))-\frac m2r(c)^2
  =m\log\Phi(c).
\]
Thus a maximizer of $\cJ_c$ yields both the compatible endpoints and the
benchmark in \eqref{eq:tilt-reduction-target}.  The next theorem establishes
existence and uniqueness and records these two outputs.

\begin{theorem}\label{thm:variational}
For every $c\in\R$, the functional $\cJ_c$ has a unique maximizer
$(q_*,v_*)\in\mathcal D$.  There are unique $s_{*,i}\in\R$ such that
\[
  H(s_{*,i})=q_*+\alpha\ip{x_i}{v_*}.
\]
Define
\[
  a_{*, i}=r(s_{*,i}),
  \qquad
  b_{*,i}=H(s_{*,i})=s_{*,i}+a_{*,i}.
\]
Then
\begin{equation}\label{eq:variational-compatibility}
  b_{*}-R a_{*}=c\one
\end{equation}
and
\begin{equation}\label{eq:variational-bound}
  \sum_{i=1}^m\log\Phi(s_{*,i})
  +\frac12a_{*}^{\trans}(I_m-R)a_{*}
  \geq m\log\Phi(c).
\end{equation}
\end{theorem}

\begin{proof}
Each term involving $\cF$ in \eqref{eq:Jc-def} is concave, while the two
negative quadratic terms are jointly strictly concave.  Hence $\cJ_c$ is
strictly concave.  Moreover, $\cF<0$ and negativity of the quadratic terms imply
$\cJ_c(q,v)\to-\infty$ as $\|(q,v)\|\to\infty$, whereas
$\cF(y)\to-\infty$ as $y\downarrow0$ implies that $\cJ_c$ tends to
$-\infty$ at the boundary of $\mathcal D$.  Its nonempty upper level sets are
therefore compact subsets of $\mathcal D$.  A maximizer exists, and strict
concavity makes it unique.

Set
\[
  b_{*,i}=q_*+\alpha\ip{x_i}{v_*}=H(s_{*,i}),
  \qquad
  a_{*,i}=\cF'(b_{*,i})=r(s_{*,i})>0.
\]
The stationarity equations \eqref{eq:Jc-gradients} give
\begin{equation}\label{eq:variational-stationarity}
  v_*=\sum_{i=1}^m a_{*,i} x_i,
  \qquad
  q_*=c+\bar a_{*},
  \qquad
  \bar a_{*}=\frac1m\sum_{i=1}^m a_{*,i}.
\end{equation}
Combining \eqref{eq:Ra-decomposition} and
\eqref{eq:variational-stationarity},
\[
  (Ra_{*})_i=\alpha\ip{x_i}{v_*}+\bar a_{*},
\]
and therefore
\[
  b_{*,i}=q_*+\alpha\ip{x_i}{v_*}=c+(Ra_{*})_i.
\]
This proves \eqref{eq:variational-compatibility}.

At the maximizer,
\[
  \sum_{i=1}^m\cF(b_{*,i})
  =\sum_{i=1}^m\log\Phi(s_{*,i})+\frac12\|a_{*}\|^2.
\]
Moreover, \eqref{eq:variational-stationarity} gives
\begin{align*}
  a_{*}^{\trans}Ra_{*}
  &=\alpha\left\|\sum_{i=1}^m a_{*,i} x_i\right\|^2
    +\frac1m\left(\sum_{i=1}^m a_{*,i}\right)^2\\
  &=\alpha\|v_*\|^2+m(q_*-c)^2.
\end{align*}
Consequently,
\begin{equation}\label{eq:Jc-at-maximizer}
  \cJ_c(q_*,v_*)
  =\sum_{i=1}^m\log\Phi(s_{*,i})
   +\frac12a_{*}^{\trans}(I_m-R)a_{*}.
\end{equation}
Finally,
\begin{align*}
  \cJ_c(H(c),0)
  &=m\left(\log\Phi(c)+\frac12r(c)^2\right)
    -\frac m2(H(c) - c)^2\\
  &=m\left(\log\Phi(c)+\frac12r(c)^2\right)
    -\frac m2r(c)^2\\
  &=m\log\Phi(c).
\end{align*}
Maximality and \eqref{eq:Jc-at-maximizer} prove
\eqref{eq:variational-bound}.
\end{proof}

\Cref{thm:variational} has produced exactly the vectors $a$ and $b$ required
by \eqref{eq:factor-first-moment} and
\eqref{eq:threshold-compatibility}, together with the scalar estimate needed in
\eqref{eq:tilt-reduction-target}.  The semidefinite hypothesis $R\succeq J/m$
is used precisely in this construction: it makes
$G=\alpha^{-1}(R-J/m)$ positive semidefinite and hence provides the Gram
representation in \eqref{eq:Ra-decomposition}.  No inverse of $R$ is used, so
the argument applies without change when $R$ is singular.

\section{Proof of the stochastic comparison}\label{sec:orthant-proof}

At this point the construction is complete.  The product inequality bounds the
weighted lower-orthant event, the Gaussian change of measure removes the
weight, and the variational theorem supplies both the common threshold and the
remaining scalar estimate.  Their combination proves \Cref{thm:orthant}.

\begin{proof}[Proof of the stochastic comparison in \Cref{thm:orthant}]
Let $R\succeq J/m$, let $X\sim\cN(0,R)$, and fix $c\in\R$.  Take
$s\in\R^m$ and $a,b\in(0,\infty)^m$ from \Cref{thm:variational}.  For
$1\leq i\leq m$, define univariate functions
\[
  f_i(z)=e^{a_i z}\ind_{(-\infty,b_i]}(z).
\]
These functions are bounded and log-concave, and
\Cref{lem:local-centering} gives
\[
  \int zf_i(z)\,d\gamma_1(z)=0,
  \qquad
  \int f_i\,d\gamma_1=e^{a_i^2/2}\Phi(s_i).
\]
Applying \Cref{thm:centered-product} yields
\begin{equation}\label{eq:product-applied}
  \E\!\left[e^{a^{\trans}X}\ind_{\{X\leq b\}}\right]
  \geq e^{\|a\|^2/2}\prod_{i=1}^m\Phi(s_i).
\end{equation}
On the other hand, \eqref{eq:variational-compatibility} and
\Cref{lem:gaussian-shift} give the event identity
\begin{equation} \label{eq:event-identity}
  \E\!\left[e^{a^{\trans}X}\ind_{\{X\leq b\}}\right]
  =e^{a^{\trans}Ra/2}\Pp\{X\leq c\one\}.
\end{equation}
The quotient of the two Gaussian normalizing factors is therefore
$\exp\{a^{\trans}(I_m-R)a/2\}$, exactly the quadratic term in the
variational bound.  Combining \eqref{eq:product-applied} and
\eqref{eq:event-identity}, and then applying
\eqref{eq:variational-bound},
\begin{align*}
  \Pp\{X\leq c\one\}
  &\geq
  \exp\!\left(\frac12a^{\trans}(I_m-R)a\right)
  \prod_{i=1}^m\Phi(s_i)\\
  &=\exp\!\left(
       \sum_{i=1}^m\log\Phi(s_i)
       +\frac12a^{\trans}(I_m-R)a
     \right)\\
  &\geq\Phi(c)^m.
\end{align*}
Since $c$ was arbitrary, this proves
$\M_X\leq_{\mathrm{st}}\M_Z$.
\end{proof}

If $R=I_m$, then equality in distribution is immediate.  If $R\neq I_m$,
the strict refinement in \Cref{thm:strict-orthant-refinement} gives
\[
  \Pp\{X\leq c\one\}>\Phi(c)^m
  \qquad(c\in\R),
\]
so the laws of $\M_X$ and $\M_Z$ are different.  This proves the equality
statement in \Cref{thm:orthant}.

The stochastic comparison, including its equality statement, is now complete.

\section{Applications}\label{sec:consequences}

\Cref{thm:orthant} compares the maximum associated with the transformed
covariance $R$ to the independent Gaussian maximum.  The applications require
two further steps.  First, stochastic order must be converted into expectation
inequalities.  Second, \eqref{eq:R-from-G} and \eqref{eq:normalized-mgf} must
be used to return from $R$ to the original score covariance $G$.  We carry this
out first for Gaussian maxima and classification, then extend the
expected-maximum comparison to covariance matrices with marginal variances at
most one, and finally treat the finite-energy AWGN sequence model.

\subsection{Increasing functions of the normalized maximum}

\begin{proof}[Proof of \Cref{cor:monotone}]
For a real random variable $U$, let
\[
  Q_U(u)=\inf\{x\in\R:\Pp\{U\leq x\}\geq u\},
  \qquad 0<u<1,
\]
be its generalized quantile function.  If $W$ is uniform on $(0,1)$, then
$Q_{\M_X}(W)$ and $Q_{\M_Z}(W)$ have the laws of $\M_X$ and $\M_Z$.
The stochastic order in \Cref{thm:orthant} is equivalent to
\[
  Q_{\M_X}(u)\leq Q_{\M_Z}(u)
  \qquad(0<u<1).
\]
Applying a nondecreasing nonnegative function $\psi$ and then taking
expectations proves the first assertion, including the case of infinite
expectations.

If $R\neq I_m$, the equality statement in \Cref{thm:orthant} shows that the
two maximum laws are different.  Their quantile functions therefore differ on
a set of positive Lebesgue measure.  For a strictly increasing $\psi$, the
inequality
\[
  \psi(Q_{\M_X}(W))\leq\psi(Q_{\M_Z}(W))
\]
is then strict with positive probability.  Integrability yields the strict
expectation inequality.  Gaussian maxima have finite moments and finite
positive exponential moments of every order, so the choices
$\psi(x)=x$ and $\psi(x)=e^{\lambda x}$ give the final statements.
\end{proof}

\subsection{Returning to the score covariance}

Let $G$ be a correlation matrix, let $\xi\sim\cN(0,G)$, and define $R$ and
$X$ by \eqref{eq:R-from-G} and \eqref{eq:X-common}.  The stochastic comparison
acts on $\max_iX_i$, but the normalized identity
\eqref{eq:normalized-mgf} transfers every positive exponential moment exactly
back to $\max_i\xi_i$.  This is precisely the transfer needed for the decoding
objective.

\begin{proof}[Proof of \Cref{cor:mgf}]
Fix $\lambda>0$ and set $\mu=\lambda/\sqrt\alpha$.  By
\Cref{cor:monotone},
\[
  \E e^{\mu\max_iX_i}\leq\E e^{\mu\M_Z},
\]
with strict inequality unless $R=I_m$.  Apply
\eqref{eq:normalized-mgf} to $G$.  Applying the same identity to
$G_\triangle$, whose transformed covariance is $I_m$, converts the right-hand
side into the regular-simplex score maximum.  The common factor
$e^{-\lambda^2/2}$ cancels, giving \eqref{eq:mgf-comparison}.  Finally,
$R=I_m$ if and only if $G=G_\triangle$, so the strict comparison gives the
stated equality characterization.
\end{proof}

The maximum-likelihood identity now turns the Gaussian MGF comparison into the
coding theorem without any additional estimate.

\begin{proof}[Proof of \Cref{cor:wsc}]
Let $G$ be the Gram matrix of the code.  By \eqref{eq:ml-integral}, its
correct-decoding probability is
\[
  \frac{e^{-\lambda^2/2}}{m}
  \E\exp\!\left\{\lambda\max_i\xi_i\right\}.
\]
The prefactor is independent of the code, so \Cref{cor:mgf} proves maximality
of $G_\triangle$.  When $m=d+1$, this matrix is realized by a regular simplex
in $\R^d$.  Equality at the fixed positive signal strength forces
$G=G_\triangle$.

The matrix $G_\triangle$ has rank $d$, so any ordered configuration with this
Gram matrix spans $\R^d$.  Two spanning ordered configurations with the same
Gram matrix differ by an orthogonal transformation.  Since relabeling the
vertices leaves $G_\triangle$ unchanged, this gives uniqueness up to
orthogonal congruence and relabeling.
\end{proof}

\subsection{Expected maxima under marginal variance constraints}

For expected maxima, unlike positive exponential moments, the common Gaussian
variable disappears after taking expectations.  Moreover, points in a unit
ball produce Gaussian coordinates whose variances may be strictly smaller
than one.  The following elementary variance-completion step bridges that
gap.

\begin{lemma}[Variance completion]\label[lemma]{lem:variance-completion}
Let $S$ be an $m\times m$ positive-semidefinite matrix with $S_{ii}\leq1$,
and let $Y\sim\cN(0,S)$.  Define
\[
  D=\diag(1-S_{11},\ldots,1-S_{mm}),
\]
and let $W\sim\cN(0,D)$ be independent of $Y$.  Then $G=S+D$ is a
correlation matrix and
\begin{equation}\label{eq:variance-completion}
  \E\max_iY_i\leq\E\max_i(Y_i+W_i).
\end{equation}
\end{lemma}

\begin{proof}
The matrix $D$ is positive semidefinite, so $G=S+D$ is positive semidefinite;
its diagonal entries are one.  For each realization of $Y$, convexity of the
maximum gives the conditional Jensen inequality
\[
  \max_iY_i
  =\max_i\E[Y_i+W_i\mid Y]
  \leq\E\!\left[\max_i(Y_i+W_i)\mid Y\right].
\]
Taking expectations proves \eqref{eq:variance-completion}.
\end{proof}

\begin{proof}[Proof of \Cref{cor:mean-width}]
We first prove the Gaussian statement.  Apply
\Cref{lem:variance-completion} and write
$\widetilde Y=Y+W\sim\cN(0,G)$.  The common-Gaussian representation
\eqref{eq:X-common}, together with $\E B=0$, gives
\begin{equation}\label{eq:expected-max-transform}
  \E\M_X=\sqrt\alpha\,\E\max_i\widetilde Y_i
\end{equation}
for the transformed covariance associated with $G$.  The same identity for
$G_\triangle$ gives
\[
  \E\M_Z=\sqrt\alpha\,\E\max_i\xi_i^\triangle.
\]
The expectation comparison in \Cref{cor:monotone} therefore yields
\[
  \E\max_iY_i
  \leq\E\max_i\widetilde Y_i
  \leq\E\max_i\xi_i^\triangle.
\]
Since the transformed covariance of $G_\triangle$ is $I_m$, the preceding
identity also gives
\[
  \E\max_i\xi_i^\triangle
  =\alpha^{-1/2}\E\M_Z
  =\sqrt{\frac{m}{m-1}}\,\E\max_iZ_i,
\]
which proves \eqref{eq:expected-max}.

It remains to identify equality.  If equality holds in
\eqref{eq:expected-max}, then the second inequality in the preceding chain is
an equality.  The strict part of \Cref{cor:monotone} therefore forces
$G=G_\triangle$.  Since $D$ is diagonal, this already gives
\[
  S_{ij}=-\frac1{m-1}
  \qquad(i\neq j).
\]
Positive semidefiniteness of $S$ and the bounds $S_{ii}\leq1$ imply
\[
  0\leq\one^{\trans}S\one
  =\sum_{i=1}^m S_{ii}
    +2\sum_{i<j}S_{ij}
  =\sum_{i=1}^mS_{ii}-m
  \leq0.
\]
Thus $S_{ii}=1$ for every $i$, so $D=0$ and $S=G_\triangle$.  Conversely,
$S=G_\triangle$ clearly gives equality.

Now let $y_1,\ldots,y_m$ lie in the Euclidean unit ball of $\R^d$, and put
$K=\conv\{y_1,\ldots,y_m\}$.  If $g\sim\cN(0,I_d)$, then the Gaussian vector
$Y_i=\ip{g}{y_i}$ has covariance
$S=(\ip{y_i}{y_j})$, which is positive semidefinite with $S_{ii}\leq1$.
Hence
\[
  \E h_K(g)
  =\E\max_i\ip{g}{y_i}
  \leq\E\max_i\xi_i^\triangle.
\]
Let $\sigma_{d-1}$ be normalized rotation-invariant measure on
$S^{d-1}$ and use the convention
\[
  w(K)=\int_{S^{d-1}}\bigl(h_K(u)+h_K(-u)\bigr)
       \,d\sigma_{d-1}(u)
      =2\int_{S^{d-1}}h_K(u)\,d\sigma_{d-1}(u).
\]
Writing $g$ as the product of its norm and its independent uniform direction
gives
\begin{equation}\label{eq:gaussian-mean-width}
  \E h_K(g)=\frac{\E\norm{g}}{2}\,w(K).
\end{equation}
Thus the Gaussian expected-maximum inequality is exactly the desired
mean-width inequality in a fixed ambient dimension.  Equality forces the Gram
matrix of the $y_i$ to be $G_\triangle$, so the points lie on the unit sphere
and form a centered regular simplex.  When $d\geq m-1$, such a simplex is
feasible, and equal Gram matrices give uniqueness up to an orthogonal map on
their spans, extended to $\R^d$, and relabeling.
\end{proof}

\subsection{Finite-energy AWGN codes}

The variance-completion argument above is sufficient for expected maxima, but
it does not by itself solve the finite-energy decoding problem.  When the
codeword norms are unequal, the Gaussian likelihood scores contain the
deterministic offsets $-\|c_i\|^2/N_0$ in addition to their random linear
parts.  The maximal per-codeword energy constraint nevertheless permits an
exact reduction to the equal-energy setting.  Because the sequence model also
requires a measure-theoretic finite-dimensional reduction, we give the full
details in \Cref{app:awgn} and retain here the structure of the argument.

\begin{proof}[Proof of \Cref{cor:finite-energy}]
The sufficient-statistic construction in \Cref{app:awgn} shows that every
$M$-message code in the AWGN sequence model is equivalent, for optimal
decoding, to a Gaussian shift experiment of dimension at most $M-1$.  Next,
one extra signal/noise coordinate is appended to replace each codeword $c_i$
by
\[
  \widetilde c_i
  =\left(c_i,\sqrt{E-\norm{c_i}^2}\right).
\]
All augmented codewords have squared norm $E$, and a decoder may ignore the
new coordinate, so this operation cannot reduce the optimal success
probability.  Projecting away their common component orthogonal to the
difference span leaves an equal-radius configuration in dimension at most
$M-1$.  After embedding that span in $\R^{M-1}$, \Cref{cor:wsc} bounds the
configuration by a regular simplex of the same radius.  Finally, a Gaussian
degradation argument shows that the optimal success probability for a regular
simplex is nondecreasing in its radius, so the full radius $\sqrt E$ is
optimal.  These reductions are formalized in \Cref{prop:awgn-reduction}.
Su et al.\ track equality through this chain to obtain the optimizer
characterization stated after \Cref{cor:finite-energy}
\cite[Theorem~3(b)]{SuYangXuI2026}.

To evaluate the optimum, translate the centered regular simplex to the
orthogonal representatives $Ae_i$, where
$A=\sqrt{EM/(M-1)}$.  Translation by a common vector does not change any
likelihood ratio.  Conditioning on the noise in the transmitted coordinate
then gives
\[
  P_{\mathrm c}^*(E,M)
  =\E\Phi\!\left(
       W+\sqrt{\frac{2EM}{(M-1)N_0}}
     \right)^{M-1},
  \qquad W\sim\cN(0,1),
\]
which is \eqref{eq:pc-opt}.  The definition of $M^*(E,\epsilon)$ now gives
\eqref{eq:M-star}.  As proved in \Cref{app:awgn}, $p_M$ is continuous and
strictly increasing, with $p_M(0)=1/M$ and $p_M(a)\uparrow1$.  Inverting
\eqref{eq:pc-opt} therefore gives \eqref{eq:E-min}, including the endpoint
cases $\epsilon\geq1-1/M$ and $\epsilon=0$.
\end{proof}

The conclusion is specific to deterministic, equiprobable, no-feedback codes
under a maximal per-codeword total-energy constraint, with no restriction on
the number of channel uses.  It proves the fixed-$(M,\epsilon)$ simplex-code
assertion discussed in \cite[Remark~21.2]{Polyanskiy_Wu_2025}.  Average energy
over the codebook, feedback, unequal priors, and non-Gaussian noise are not
covered by this reduction.

This completes the applications stated in \Cref{sec:results}.  We close by
recording the equality-case question posed in version~1, now resolved, followed
by three further open problems.

\section{A resolved problem and further questions}\label{sec:discussion}

The first version of this paper posed the following question.

\begin{resolvedproblem}[Equality cases---resolved]
\label[resolvedproblem]{prob:equality-cases-resolved}
Determine the equality cases in \eqref{eq:orthant-main},
\eqref{eq:mgf-comparison}, and the coding and mean-width consequences.  In
particular, determine when equality at one nontrivial threshold or one positive
MGF parameter forces $R=I_m$, or equivalently forces the original covariance
to be $G_\triangle$ after \eqref{eq:R-from-G}.
\end{resolvedproblem}

This statement appeared as Open Problem~8.1 in version~1.  It was resolved
independently by the present author and by Su et al.~\cite{SuYangXuI2026}.  The present manuscript proves the stronger
pointwise statement in \Cref{thm:strict-orthant-refinement};
\Cref{cor:mgf,cor:wsc,cor:mean-width} give the corresponding MGF, coding, and
mean-width rigidity statements.  We retain the statement and number as a historical resolved problem to preserve the referent in the title of \cite{SuYangXuI2026}.  The chronology and the relation between the two proof
routes are described in \Cref{sec:prior-work}.

The common-threshold comparison also isolates several nearby problems for
which the present proof supplies only an upper bound or a counterexample to
the most naive extension.

\begin{problem}[The rank-constrained Gaussian maximum]\label[problem]{prob:rank-constrained}
Fix integers $m\geq3$ and $1\leq d<m-1$, and fix $\lambda>0$.  Determine
\[
  \max\left\{
    \E\exp\!\left(\lambda\max_{1\leq i\leq m}\xi_i\right):
    \xi\sim\cN(0,G),\ G\text{ a correlation matrix},\ \rank G\leq d
  \right\},
\]
and characterize all maximizing covariance matrices.  In particular, determine
whether the maximizing configuration depends on $\lambda$.
\end{problem}

The feasible set is compact, so a maximizer exists.  \Cref{cor:mgf} gives the
regular $(m-1)$-simplex value as a universal upper bound, but that covariance
has rank $m-1$ and is not feasible.  Thus this problem asks for the optimal
overloaded equal-energy Gaussian code when the number of signals exceeds the
ambient dimension plus one.  The unconstrained theorem does not indicate which
boundary point of the rank-constrained elliptope should be optimal, so even
the candidate geometry remains to be identified.

\begin{problem}[Universal lower-orthant thresholds]\label[problem]{prob:threshold-vectors}
Determine all vectors $t=(t_1,\ldots,t_m)\in\R^m$ for which
\[
  \Pp\{X_i\leq t_i\text{ for every }i\}
  \geq\prod_{i=1}^m\Phi(t_i)
\]
holds for every correlation matrix $R\succeq J/m$ and every
$X\sim\cN(0,R)$.
\end{problem}

\Cref{thm:orthant} shows that every constant vector $t=c\one$ has this
property, whereas \Cref{rem:unequal-thresholds} shows that arbitrary threshold
vectors do not.  A solution would identify the precise reach of the
common-threshold phenomenon.  It would also clarify whether the adaptive
centering construction can be modified to treat any genuinely asymmetric
family of one-sided events.

\begin{problem}[Average-energy Gaussian codes]\label[problem]{prob:average-energy}
For fixed $M$, ambient dimension, noise level, and average-energy budget $E$,
maximize the average probability of correct decoding over equiprobable
Gaussian codebooks satisfying
\[
  \frac1M\sum_{i=1}^M\norm{c_i}^2\leq E.
\]
Characterize the optimizers as a function of the signal-to-noise ratio.  Can an
optimizer always be chosen from a family consisting of a lower-dimensional
regular simplex together with coincident zero signals, or are additional
geometries unavoidable?
\end{problem}

The regular simplex is not always optimal under this weaker constraint: this
is the failure of the Strong Simplex Conjecture proved by Steiner
\cite{Steiner1994}.  The known counterexample literature already exhibits a
substantial family of lower-dimensional degeneracies.  In particular, Sun
studied codebooks formed by a regular simplex on $K$ nonzero signals together
with $M-K$ zero signals, as well as related circle-plus-zero configurations
\cite{Sun1997}.  These results show that the failure is not confined to a
single exceptional construction, but they do not give a general description
of the global optimizer.  The maximal per-codeword constraint used in
\Cref{cor:finite-energy} rules out precisely this redistribution of energy
among messages.

Other extensions---including unequal priors, feedback, and non-Gaussian
noise---change either the one-dimensional centering equations or the Gaussian
change of measure and therefore require new ideas beyond the argument given
here.

\section*{Tool and computational resource disclosure}

Large language models provided by OpenAI, including GPT-5.6 Pro and GPT-5.6
Sol Max (accessed in July 2026), were used during the development of this work
to assist with mathematical exploration, proof development and checking,
literature discovery and organization, drafting, and editorial revision.  The
author independently checked the mathematical arguments and references and
takes full responsibility for the correctness, attribution, and presentation
of the paper.  No theorem, proof step, or bibliographic assertion was retained
solely on the authority of an automated system.  No automated system is an
author, and the proofs do not rely on a computer-assisted verification.

\clearpage
\appendix
\section{Calculus for the truncated exponential functions}
\label{app:truncated-functions}

This appendix proves the properties of $r$, $H$, and $\cF$ used in
\Cref{sec:one-d-calc}.

Let $Z\sim\cN(0,1)$. The identity \(\varphi'(z) = - z \varphi(z)\) gives
\begin{equation}\label{eq:truncated-first-moment}
  \int_{-\infty}^s z\varphi(z)\,dz=-\varphi(s) = -r(s) \Phi(s),
\end{equation}
and therefore
\[
  r(s)=\frac{\varphi(s)}{\Phi(s)}
      =-\E[Z\mid Z\leq s] > 0,
  \qquad
  H(s)=s+r(s)=\E[s-Z\mid Z\leq s]>0.
\]
Differentiating $r$ yields
\begin{equation}\label{eq:r-prime}
  r'(s)=-r(s)H(s)<0.
\end{equation}
Thus $r$ is strictly decreasing.  It tends to zero as $s\to\infty$.
Moreover, $Z\leq s$ under the conditioning in the preceding display, so
$r(s)\geq-s$ and hence $r(s)\to\infty$ as $s\to-\infty$.  Consequently,
$r$ maps $\R$ bijectively onto $(0,\infty)$.

The identity \((z \varphi(z))' = \varphi(z) - z^{2}\varphi(z)\) gives
\[
  \int_{-\infty}^s z^2\varphi(z)\,dz
  =\Phi(s)-s\varphi(s) = \Phi(s)(1 - s r(s)).
\]
It follows that
\begin{equation}\label{eq:H-prime}
  H'(s)
  =1-sr(s)-r(s)^2
  =\operatorname{Var}(Z\mid Z\leq s)>0.
\end{equation}
Hence $H$ is strictly increasing.  As $s\to\infty$,
$H(s)=s+r(s)\to\infty$.  At the other endpoint,
\[
  H(s)\Phi(s)
  =s\Phi(s)+\varphi(s)
  =\int_{-\infty}^s\Phi(u)\,du.
\]
Both the numerator and denominator in
\[
  H(s)=\frac{\int_{-\infty}^s\Phi(u)\,du}{\Phi(s)}
\]
tend to zero as $s\to-\infty$.  L'H\^opital's rule and the preceding
limit for $r$ give
\[
  \lim_{s\to-\infty}H(s)
  =\lim_{s\to-\infty}\frac{\Phi(s)}{\varphi(s)}
  =\lim_{s\to-\infty}\frac1{r(s)}=0.
\]
Thus $H$ maps $\R$ bijectively onto $(0,\infty)$.

Because $H$ is bijective, the definition
\[
  \cF(H(s))=\log\Phi(s)+\frac12r(s)^2
\]
determines a function $\cF:(0,\infty)\to\R$.  Using
\eqref{eq:r-prime} and \eqref{eq:H-prime},
\begin{align*}
  \frac{d}{ds}\left(\log\Phi(s)+\frac12r(s)^2\right)
  &=r(s)+r(s)r'(s)\\
  &=r(s)H'(s).
\end{align*}
Therefore, chain rule gives
\begin{equation}\label{eq:F-prime}
  \cF'(H(s))=r(s)>0.
\end{equation}
Differentiating once more gives
\[
  \cF''(H(s))=\frac{r'(s)}{H'(s)}<0.
\]
Hence $\cF$ is strictly increasing and strictly concave.  As $s\to\infty$,
$\Phi(s)\to1$ and $r(s)\to0$, so $\cF(H(s))\to0$.

It remains to determine the other endpoint.  Put $s=-t$, $t>0$, and write
$Q(t)=\Phi(-t)$.  The classical Mills bounds \cite{Gordon1941}
\[
  \frac{t}{1+t^2}\varphi(t)
  \leq Q(t)
  \leq\frac{\varphi(t)}t
\]
imply
\[
  r(-t)=\frac{\varphi(t)}{Q(t)}\leq t+\frac1t.
\]
Consequently,
\begin{align*}
  \cF(H(-t))
  &=\log Q(t)+\frac12r(-t)^2\\
  &\leq
    -\frac{t^2}{2}-\log t-\frac12\log(2\pi)
    +\frac12\left(t+\frac1t\right)^2\\
  &=1+\frac1{2t^2}-\log t-\frac12\log(2\pi)
  \longrightarrow-\infty.
\end{align*}
Since $H(-t)\downarrow0$, this proves
$\cF(y)\to-\infty$ as $y\downarrow0$.  Together with monotonicity and the
limit at infinity, it shows that $\cF$ maps $(0,\infty)$ bijectively onto
$(-\infty,0)$ and is negative throughout its domain.  This proves
\Cref{prop:tilt-functions}.

\section{Details on the Gaussian product inequality}
\label{app:centered-product}

This appendix first records how \Cref{thm:centered-product} follows from the
general Gaussian forward--reverse Brascamp--Lieb theorem in \Cref{app:frbl-specialization}.  It then gives an alternate direct proof from the symmetric Gaussian rectangle inequality, normalized
self-convolution, and the central limit theorem in the subsequent sections of \Cref{app:centered-product}.  The direct proof also yields
the stronger strict inequality at every finite common threshold stated in
\Cref{rem:strict-common-threshold}. More general centered Gaussian correlation
and forward--reverse Brascamp--Lieb inequalities appear in
\cite{NakamuraTsuji2025Centered,NakamuraTsuji2026Saturation,
MilmanNakamuraTsuji2026}.

\subsection{Specialization of the forward--reverse Brascamp--Lieb theorem}
\label{app:frbl-specialization}

\Cref{thm:centered-product} is a specialization of
\cite[Theorem~1.10]{MilmanNakamuraTsuji2026}.  In the notation of that theorem,
take $I=m$, $n_i=1$,
$c_i=1$, and $\Sigma_i=1$ for every input; take a single output with
$m_1=m$, $d_1=1$, $L_1=\operatorname{Id}_{\R^m}$, $\Gamma_1=R$, and $\mathcal{Q}=0$.
Set $h(x)=\prod_{i=1}^m f_i(x_i)$.  The pointwise hypothesis then holds with
equality, and the centeredness condition is precisely
\eqref{eq:zero-first-moment}.

The remaining Gaussian constant in this specialization is
\begin{equation}\label{eq:MNT-constant}
 \sup_{\substack{A_i\geq0,\ B\succeq0\\
                  \diag(A_1,\ldots,A_m)\succeq B}}
 \frac{\det(I_m+RB)^{1/2}}{\prod_{i=1}^m(1+A_i)^{1/2}}.
\end{equation}
Writing $A=\diag(A_1,\ldots,A_m)$ and using $B\preceq A$, Loewner
monotonicity of the determinant, Sylvester's determinant identity, and
Hadamard's inequality give
\begin{align*}
 \det(I_m+RB)
 &=\det(I_m+R^{1/2}BR^{1/2})\\
 &\leq\det(I_m+R^{1/2}AR^{1/2})\\
 &=\det(I_m+A^{1/2}RA^{1/2})\\
 &\leq\prod_{i=1}^m(1+A_i),
\end{align*}
because $R_{ii}=1$.  Thus the constant in \eqref{eq:MNT-constant} is at most
one; taking $A=B=0$ shows that it is exactly one.  The general theorem allows
$\Gamma_1=R$ to be singular, and hence gives \eqref{eq:centered-product} for
every correlation matrix $R$.

\subsection{Symmetric rectangles and even functions}

We begin with the classical \v{S}id\'ak--Khatri inequality in the form needed
below \cite{Khatri1967,Sidak1967}, together with its equality case for finite
nondegenerate rectangles.  For a recent refinement, see \cite{AssoulineChorSadovsky2024}. For self-containment, we provide a direct proof.

\begin{lemma}\label[lemma]{lem:rectangles}
Let $V=(V_1,\ldots,V_m)$ be a centered Gaussian vector with
$\operatorname{Var}(V_i)=1$ for every $i$.  For
$r_1,\ldots,r_m\in[0,\infty]$,
\begin{equation}\label{eq:rectangle}
  \Pp\{|V_i|\leq r_i\text{ for all }i\}
  \geq\prod_{i=1}^m\Pp\{|Z|\leq r_i\},
  \qquad Z\sim\cN(0,1).
\end{equation}
If $0<r_i<\infty$ for every $i$, then equality in
\eqref{eq:rectangle} holds if and only if
$\operatorname{Cov}(V)=I_m$.
\end{lemma}

\begin{proof}
By monotone convergence, it is enough for the non-strict assertion to treat
finite radii.  We proceed by induction on $m$ and allow singular covariance
matrices throughout.  The case $m=1$ is equality.  For the induction step,
write
\[
  T=V_m,\qquad W=(V_1,\ldots,V_{m-1}),\qquad
  \beta=\operatorname{Cov}(W,T),
\]
and set $Y=W-\beta T$.  The Gaussian vectors $Y$ and $T$ are uncorrelated and
therefore independent, even when the joint covariance is singular.  With
\[
  A=\prod_{i<m}[-r_i,r_i],
  \qquad
  q(t)=\Pp\{Y+\beta t\in A\},
\]
symmetry of $A$ and $Y$ makes $q$ even.  The function $q$ is log-concave as
well.  Indeed, represent $Y=LZ'$ on its linear support, with $Z'$ a standard
Gaussian vector of the appropriate dimension.  Then
\[
  q(t)=\int \ind_A(Lz+\beta t)\,d\gamma_k(z)
\]
for a suitable $k$, and Pr\'ekopa's theorem
\cite{Prekopa1973,BrascampLieb1976} applies to the jointly log-concave
integrand.  Hence $q$ is nonincreasing on $[0,\infty)$.

Let $U=|T|$, and let $U'$ be an independent copy.  The functions
$u\mapsto q(u)$ and $u\mapsto\ind_{\{u\leq r_m\}}$ are both nonincreasing, so
\begin{align*}
  2\operatorname{Cov}\!\left(q(U),\ind_{\{U\leq r_m\}}\right)
  =\E\!\left[
    (q(U)-q(U'))
    (\ind_{\{U\leq r_m\}}-\ind_{\{U'\leq r_m\}})
  \right]\geq0.
\end{align*}
Consequently,
\begin{align*}
 \Pp\{V\in A\times[-r_m,r_m]\}
 &=\E q(T)\ind_{\{|T|\leq r_m\}}\\
 &\geq \E q(T)\,\Pp\{|Z|\leq r_m\}.
\end{align*}
The first factor on the right is the rectangle probability for $W$, so the
induction hypothesis proves \eqref{eq:rectangle}.

It remains to prove strictness.  Assume that every radius is finite and
positive and that $R:=\operatorname{Cov}(V)\neq I_m$.  After relabeling the
coordinates, we may arrange that the vector $\beta$
above is nonzero.  Choose $\delta>0$ so that
$[-2\delta,2\delta]^{m-1}\subset A$.  For all sufficiently small $|t|$,
$\lVert\beta t\rVert_\infty\leq\delta$, and hence
$\{\lVert Y\rVert_\infty\leq\delta\}\subseteq
\{Y+\beta t\in A\}$.  A centered Gaussian assigns positive mass to this
neighborhood of the origin relative to its linear support, so $q(t)>0$ for all
sufficiently small $|t|$.  On the other hand, choose $j<m$ with
$\beta_j\neq0$.  Then
\[
  q(t)\leq\Pp\{|Y_j+\beta_jt|\leq r_j\}\longrightarrow0
  \qquad(|t|\to\infty).
\]
We can therefore choose
$0<a<r_m<b$ such that $q(a)>q(b)$.  Since $U=|T|$ has a strictly positive
density on $(0,\infty)$, the integrand in the preceding covariance identity
is strictly positive on the event $\{U<a,\ U'>b\}$, which has positive
probability.  Thus the last-coordinate inequality is strict, and so is
\eqref{eq:rectangle}.  Conversely, when $R=I_m$, the coordinates are
independent and equality is immediate.
\end{proof}

\begin{remark}\label[remark]{rem:NT-rectangle-equality}
The equality statement can also be obtained from Nakamura and Tsuji's two-set
theorem \cite[Theorem~1.3]{NakamuraTsuji2025Centered} with some additional
effort; however, the direct argument above is shorter in the present setting.
Su et al.\ independently prove the same strict rectangle statement by first
isolating a strict bivariate gap and then using Royen's Gaussian correlation
theorem to group the remaining coordinate slabs
\cite[Lemma~9]{SuYangXuI2026}.
\end{remark}

\begin{lemma}\label[lemma]{lem:even-factors}
Under the hypotheses of \Cref{lem:rectangles}, let
$g_1,\ldots,g_m:\R\to[0,\infty)$ be bounded, even, and log-concave.  Then
\begin{equation}\label{eq:even-product}
  \E\prod_{i=1}^m g_i(V_i)
  \geq\prod_{i=1}^m\int g_i\,d\gamma_1.
\end{equation}
\end{lemma}

\begin{proof}
Each strict superlevel set of an even log-concave function on $\R$ is an
interval centered at the origin, possibly empty or unbounded. Then the layer-cake formula, Tonelli's theorem,
and \Cref{lem:rectangles} give
\begin{align*}
 \E\prod_{i=1}^m g_i(V_i)
 &=\int_{[0,\infty)^m}
   \Pp\{g_i(V_i)>t_i\text{ for all }i\}\,dt_1\cdots dt_m\\
 &\geq\int_{[0,\infty)^m}
   \prod_{i=1}^m\Pp\{g_i(Z)>t_i\}\,dt_1\cdots dt_m\\
 &=\prod_{i=1}^m\E g_i(Z).
\end{align*}
\end{proof}

\subsection{Normalized self-convolution}

Normalize the functions in \Cref{thm:centered-product} by
\begin{equation}\label{eq:h-normalization}
  h_i=\frac{f_i}{\int f_i\,d\gamma_1}.
\end{equation}
Then $h_i\,d\gamma_1$ is a probability measure with mean zero.  For a
nonnegative function $h$, define
\begin{equation}\label{eq:selfconv-D-def}
  (\mathcal Dh)(u)=\int_{\R}
  h\!\left(\frac{u+v}{\sqrt2}\right)
  h\!\left(\frac{u-v}{\sqrt2}\right)d\gamma_1(v).
\end{equation}
If $Y,Y'$ are independent with law $h\,d\gamma_1$, the Gaussian
sum--difference rotation gives
\begin{equation}\label{eq:D-law}
  (\mathcal Dh)\,d\gamma_1
  =\operatorname{Law}\!\left(\frac{Y+Y'}{\sqrt2}\right).
\end{equation}
Indeed, for every bounded test function $\psi$,
\begin{align*}
 \E\psi\!\left(\frac{Y+Y'}{\sqrt2}\right)
 &=\iint \psi\!\left(\frac{y+y'}{\sqrt2}\right)
   h(y)h(y')\,d\gamma_1(y)d\gamma_1(y')\\
 &=\iint\psi(u)
   h\!\left(\frac{u+v}{\sqrt2}\right)
   h\!\left(\frac{u-v}{\sqrt2}\right)
   d\gamma_1(u)d\gamma_1(v).
\end{align*}

If $h$ is bounded and log-concave, then so is $\mathcal Dh$.  The bound
$\mathcal Dh\leq\norm{h}_\infty^2$ proves boundedness.  For log-concavity,
write the integral in
\eqref{eq:selfconv-D-def} with respect to Lebesgue measure and apply
Pr\'ekopa's theorem to
\[
  (u,v)\longmapsto
  h\!\left(\frac{u+v}{\sqrt2}\right)
  h\!\left(\frac{u-v}{\sqrt2}\right)\varphi(v).
\]
Equation \eqref{eq:D-law} shows that total mass, mean zero, and variance are
preserved.  Iterating, if $Y_1,Y_2,\ldots$ are i.i.d.\ with law
$h\,d\gamma_1$, then
\begin{equation}\label{eq:D-iteration-law}
  (\mathcal D^rh)\,d\gamma_1
  =\operatorname{Law}\!\left(
      2^{-r/2}\sum_{\ell=1}^{2^r}Y_\ell
    \right).
\end{equation}
The mean-zero assumption is essential: otherwise the mean would be multiplied
by $\sqrt2$ at each iteration.

For bounded functions $q_1,\ldots,q_m$, write
\begin{equation}\label{eq:Z-functional}
  \mathcal Z_R(q_1,\ldots,q_m)
  =\E_{\cN(0,R)}\prod_{i=1}^m q_i(X_i).
\end{equation}

\begin{lemma}\label[lemma]{lem:sum-difference}
For bounded log-concave $h_1,\ldots,h_m$ and every correlation matrix $R$,
\begin{equation}\label{eq:sum-difference}
  \mathcal Z_R(h_1,\ldots,h_m)^2
  \geq\mathcal Z_R(\mathcal Dh_1,\ldots,\mathcal Dh_m).
\end{equation}
\end{lemma}

\begin{proof}
Let $X,X'$ be independent with law $\cN(0,R)$ and put
\[
  U=\frac{X+X'}{\sqrt2},\qquad
  V=\frac{X-X'}{\sqrt2}.
\]
Then $U$ and $V$ are independent and each has law $\cN(0,R)$.  Conditional
on $U=u$, the $i$th factor in the square of \eqref{eq:Z-functional} is
\[
  g_{i,u_i}(v)
  =h_i\!\left(\frac{u_i+v}{\sqrt2}\right)
   h_i\!\left(\frac{u_i-v}{\sqrt2}\right).
\]
As a function of $v$, it is bounded, even, and log-concave.  Applying
\Cref{lem:even-factors} to $V$ conditionally on $U=u$ gives
\begin{align*}
 \E_V\prod_i g_{i,u_i}(V_i)
 &\geq\prod_i\int g_{i,u_i}(v)\,d\gamma_1(v)\\
 &=\prod_i(\mathcal Dh_i)(u_i).
\end{align*}
Averaging over $U$ proves \eqref{eq:sum-difference}.
\end{proof}

\subsection{Iteration and the central limit theorem}

Set $h_{i,0}=h_i$, $h_{i,r+1}=\mathcal Dh_{i,r}$, and
\begin{equation}\label{eq:Qr}
  Q_r=\mathcal Z_R(h_{1,r},\ldots,h_{m,r}).
\end{equation}
By \Cref{lem:sum-difference},
\begin{equation}\label{eq:deficit-amplification}
  Q_{r+1}\leq Q_r^2,
  \qquad
  Q_r\leq Q_0^{2^r}.
\end{equation}

Suppose first that $R$ is positive definite.  Let $Y_{i,\ell}$ be i.i.d.\ with
law $h_i\,d\gamma_1$.  These variables have mean zero and finite moments of all orders.  Indeed, for
every $p>0$,
\[
  \int |y|^p h_i(y)\,d\gamma_1(y)
  \leq\norm{h_i}_\infty\int |y|^p\,d\gamma_1(y)<\infty.
\]
Their variances $\sigma_i^2$ are finite and strictly positive, since a
probability measure absolutely continuous with respect to Gaussian measure
cannot be a point mass.  By
\eqref{eq:D-iteration-law} and the central limit theorem,
\begin{equation}\label{eq:clt}
  h_{i,r}\,d\gamma_1
  \Longrightarrow\cN(0,\sigma_i^2).
\end{equation}

The density of $\cN(0,R)$ with respect to $\gamma_1^{\otimes m}$ is
\begin{equation}\label{eq:LR}
  L_R(x)=(\det R)^{-1/2}
  \exp\!\left[-\frac12x^{\trans}(R^{-1}-I)x\right].
\end{equation}
For fixed $L>0$, continuity and strict positivity imply
\[
  c_{R,L}:=\min_{[-L,L]^m}L_R>0.
\]
It follows from \eqref{eq:Qr} that
\begin{align}
 Q_r
 &\geq c_{R,L}\prod_{i=1}^m
   \int_{-L}^Lh_{i,r}(x)\,d\gamma_1(x)\notag\\
 &=c_{R,L}\prod_{i=1}^m
   \Pp\!\left\{
      \left|2^{-r/2}\sum_{\ell=1}^{2^r}Y_{i,\ell}\right|\leq L
   \right\}.
 \label{eq:Qr-lower}
\end{align}
By \eqref{eq:clt}, the product on the right converges to
$\prod_i\Pp\{|\sigma_iZ|\leq L\}>0$.  Therefore
\begin{equation}\label{eq:positive-liminf}
  \liminf_{r\to\infty}Q_r>0.
\end{equation}
If $Q_0<1$, however, \eqref{eq:deficit-amplification} forces $Q_r\to0$.
This contradiction shows that $Q_0\geq1$.  Undoing the normalization
\eqref{eq:h-normalization} proves \eqref{eq:centered-product} for
positive-definite $R$.

Finally, let $R$ be singular.  Let $W\sim\cN(0,I_m)$ be independent of
$X\sim\cN(0,R)$ and define
\[
  X^{(\varepsilon)}
  =\sqrt{1-\varepsilon}\,X+\sqrt\varepsilon\,W,
  \qquad
  \operatorname{Cov}(X^{(\varepsilon)})
  =(1-\varepsilon)R+\varepsilon I_m.
\]
The positive-definite result applies for every $\varepsilon>0$.  A nonzero
one-dimensional log-concave function is continuous in the interior of its
support and may be discontinuous only at its two endpoints.  Since every
$X_i$ has a continuous standard Gaussian marginal, the bounded function
$\prod_i f_i(x_i)$ is almost surely continuous at $X$.  Dominated convergence
therefore gives
\[
  \E\prod_i f_i(X_i)
  =\lim_{\varepsilon\downarrow0}
    \E\prod_i f_i(X_i^{(\varepsilon)})
  \geq\prod_i\int f_i\,d\gamma_1.
\]
This proves \Cref{thm:centered-product} for singular correlation matrices as
well.

\subsection{Strictness for the truncated exponential factors}

Su et al.\ independently insert their Royen-based rectangle lemma into the same
first sum--difference/self-convolution step to obtain the strict product
inequality for the truncated exponential factors
\cite[Lemma~10]{SuYangXuI2026}.  The argument below instead uses the direct
strict rectangle theorem \Cref{lem:rectangles}; after this strict first step,
both proofs invoke the non-strict centered product inequality for the convolved
factors.

\begin{proposition}\label[proposition]{prop:strict-tilted-product}
Let $R$ be an $m\times m$ correlation matrix with $R\neq I_m$, and let
$X\sim\cN(0,R)$.  For each $i$, let
\[
  f_i(z)=e^{a_i z}\ind_{(-\infty,b_i]}(z),
  \qquad a_i>0,\quad b_i>0,
\]
and suppose that $\int zf_i(z)\,d\gamma_1(z)=0$.  Then
\begin{equation}\label{eq:strict-tilted-product}
  \E\prod_{i=1}^m f_i(X_i)
  >\prod_{i=1}^m\int f_i\,d\gamma_1.
\end{equation}
The matrix $R$ may be singular.
\end{proposition}

\begin{proof}
Normalize the factors as in \eqref{eq:h-normalization}, and write
\[
  Q_0=\mathcal Z_R(h_1,\ldots,h_m).
\]
Let $U,V$ be independent with law $\cN(0,R)$.  In the sum--difference
calculation from \Cref{lem:sum-difference}, conditioning on $U=u$ gives
\begin{align*}
 g_{i,u_i}(v)
 &=h_i\!\left(\frac{u_i+v}{\sqrt2}\right)
   h_i\!\left(\frac{u_i-v}{\sqrt2}\right)\\
 &=C_i(u_i)\,
   \ind_{\{|v|\leq\sqrt2 b_i-u_i\}},
\end{align*}
where $C_i(u_i)>0$; if the displayed radius is negative, the interval is
understood to be empty.  Consider
\[
  E=\{u\in\R^m:u_i<\sqrt2 b_i\text{ for every }i\}.
\]
Because every $b_i$ is positive, $E$ contains a neighborhood of the origin.
Every centered Gaussian measure assigns positive mass to such a neighborhood
relative to its linear support, so $\Pp\{U\in E\}>0$, including when $R$ is
singular.

For every $u\in E$, all the radii $\sqrt2b_i-u_i$ are finite and strictly
positive.  Since $R\neq I_m$, the strict part of \Cref{lem:rectangles} gives
\[
  \E_V\prod_i g_{i,u_i}(V_i)
  >\prod_i\int g_{i,u_i}\,d\gamma_1.
\]
For arbitrary $u$, the corresponding non-strict inequality follows from
\Cref{lem:even-factors}.  Integrating over $U$ therefore makes the
sum--difference inequality strict:
\begin{equation}\label{eq:strict-sum-difference}
  Q_0^2>
  \mathcal Z_R(\mathcal Dh_1,\ldots,\mathcal Dh_m).
\end{equation}
The functions $\mathcal Dh_i$ are bounded and log-concave, and their Gaussian
integrals equal one.  By \eqref{eq:D-law}, their first moments under
$\gamma_1$ are also zero.  Applying \Cref{thm:centered-product} to these
functions shows that the right-hand side of
\eqref{eq:strict-sum-difference} is at least one.  Hence
$Q_0>1$.  Undoing the normalization proves
\eqref{eq:strict-tilted-product}.
\end{proof}

\begin{theorem}
\label[theorem]{thm:strict-orthant-refinement}
Under the hypotheses of \Cref{thm:orthant}, if $R\neq I_m$, then
\[
  \Pp\{X\leq c\one\}>\Phi(c)^m
  \qquad\text{for every finite }c\in\R.
\]
\end{theorem}

\begin{proof}
Fix $c\in\R$ and take the vectors $s,a,b$ supplied by
\Cref{thm:variational}; in particular,
\[
  a_i=r(s_i)>0,
  \qquad b_i=s_i+a_i=H(s_i)>0,
  \qquad b-Ra=c\one.
\]
The factors in \eqref{eq:tilted-factors} are centered by
\eqref{eq:factor-first-moment}.  Since $R\neq I_m$,
\Cref{prop:strict-tilted-product} strengthens
\eqref{eq:product-applied} to a strict inequality.  Combining it with the
Gaussian shift identity \eqref{eq:event-shift} gives
\begin{align*}
 \Pp\{X\leq c\one\}
 &>
 \exp\!\left(\frac12a^{\trans}(I-R)a\right)
 \prod_{i=1}^m\Phi(s_i)\\
 &=\exp\!\left(
   \sum_{i=1}^m\log\Phi(s_i)
   +\frac12a^{\trans}(I-R)a
 \right)\\
 &\geq\Phi(c)^m,
\end{align*}
where the last inequality is \eqref{eq:variational-bound}.  This argument
applies directly to singular $R$, and $c$ was arbitrary.
\end{proof}

\section{Rigidity of the independent Gaussian maximum law}
\label{app:maximum-rigidity}

The equality statement in \Cref{thm:orthant} follows in the main text from the
strict product inequality tailored to the truncated exponential factors.
Dai and Mukherjea studied the broader identification problem for multivariate
normal families with a common correlation parameter and also for a
common-covariance family \cite{DaiMukherjea2001}.  The proposition below treats
a different comparison: equality with the independent standard maximum for an
arbitrary covariance matrix, including singular matrices and unequal marginal
variances.

Su et al.\ independently developed and machine-checked a related upper-tail
rigidity route under the normalized correlation-matrix hypotheses of the main
theorem \cite[Remark~1]{SuYangXuI2026}.  Their route passes from stochastic
domination and equality of one positive exponential moment to equality of the
maximum laws, and then uses bivariate upper-tail asymptotics to force $R=I_m$;
their formalization also proves Gram rigidity.  The argument below was
developed independently and requires neither unit marginal variances nor the
semidefinite condition $R\succeq J/m$.

Write $\overline\Phi(t)=1-\Phi(t)$.

\begin{lemma}\label[lemma]{lem:bivariate-upper-tails}
Let $(G,H)$ be a centered bivariate Gaussian vector with unit variances and
correlation $\rho\in[-1,1]$.  Then
\[
  \lim_{t\to\infty}
  \frac{\Pp\{G>t,H>t\}}{\overline\Phi(t)^2}
  =
  \begin{cases}
    0,&\rho<0,\\
    1,&\rho=0,\\
    +\infty,&\rho>0.
  \end{cases}
\]
\end{lemma}

\begin{proof}
The case $\rho=0$ is independence.  If $-1<\rho<0$, then
$\{G>t,H>t\}\subseteq\{G+H>2t\}$ and
$G+H\sim\cN(0,2(1+\rho))$.  Hence
\[
  \Pp\{G>t,H>t\}
  \leq
  \overline\Phi\!\left(\sqrt{\frac{2}{1+\rho}}\,t\right).
\]
Mills' bounds show that the right-hand side is
$o(\overline\Phi(t)^2)$.  For $\rho=-1$, the joint exceedance probability is
zero for $t>0$.

Now suppose $0<\rho<1$.  The joint density is
\[
  f_\rho(x,y)
  =\frac{1}{2\pi\sqrt{1-\rho^2}}
    \exp\!\left(
      -\frac{x^2-2\rho xy+y^2}{2(1-\rho^2)}
    \right).
\]
Integrate over $[t,t+t^{-1}]^2$.  Writing
$x=t+u/t$ and $y=t+v/t$, where $u,v\in[0,1]$, gives, uniformly on this square,
\[
  \frac{x^2-2\rho xy+y^2}{2(1-\rho^2)}
  =\frac{t^2}{1+\rho}+O_\rho(1).
\]
Thus, for all sufficiently large $t$,
\[
  \Pp\{G>t,H>t\}
  \geq c_\rho t^{-2}
       \exp\!\left(-\frac{t^2}{1+\rho}\right).
\]
Since Mills' upper bound gives
$\overline\Phi(t)^2\leq(2\pi t^2)^{-1}e^{-t^2}$, the ratio tends to
infinity.  If $\rho=1$, then $G=H$ almost surely and the ratio equals
$1/\overline\Phi(t)\to\infty$.
\end{proof}

\begin{proposition}\label[proposition]{prop:maximum-law-rigidity}
Let $Y=(Y_1,\ldots,Y_m)$ be any centered Gaussian vector with covariance
matrix $S$, and let $Z_1,\ldots,Z_m$ be independent standard Gaussian random
variables.  Then
\[
  \max_iY_i\stackrel{d}{=}\max_iZ_i
  \qquad\Longleftrightarrow\qquad
  S=I_m.
\]
\end{proposition}

\begin{proof}
Only the forward implication requires proof.  Put
$p_t=\overline\Phi(t)$ and $M_Y=\max_iY_i$.  Equality of the maximum laws gives
\begin{equation}\label{eq:maximum-tail-first-order}
  \frac{\Pp\{M_Y>t\}}{p_t}
  =\frac{1-(1-p_t)^m}{p_t}
  \longrightarrow m.
\end{equation}
Let $\sigma_i^2=S_{ii}$.  Mills' bounds imply, for every $\sigma\geq0$,
\[
  \frac{\overline\Phi(t/\sigma)}{\overline\Phi(t)}
  \longrightarrow
  \begin{cases}
    0,&0\leq\sigma<1,\\
    1,&\sigma=1,\\
    +\infty,&\sigma>1,
  \end{cases}
\]
where the numerator is understood to be zero for $t>0$ when $\sigma=0$.
If some $\sigma_i>1$, then
$\Pp\{M_Y>t\}\geq\Pp\{Y_i>t\}$, contradicting
\eqref{eq:maximum-tail-first-order}.  Hence every $\sigma_i\leq1$.  The union
bound and \eqref{eq:maximum-tail-first-order} now give
\[
  m
  \leq\#\{i:\sigma_i=1\}
  \leq m.
\]
Thus all marginal variances equal one.

For $t>0$, set $A_i(t)=\{Y_i>t\}$ and
$N_t=\sum_i\ind_{A_i(t)}$.  Since $\E N_t=mp_t$,
\begin{equation}\label{eq:overlap-deficit}
  D_t:=mp_t-\Pp\{M_Y>t\}=\E(N_t-1)_+.
\end{equation}
Equality of the maximum laws yields
\begin{equation}\label{eq:overlap-deficit-limit}
  \frac{D_t}{p_t^2}
  =\frac{mp_t-[1-(1-p_t)^m]}{p_t^2}
  \longrightarrow\binom m2.
\end{equation}
For every integer $0\leq k\leq m$,
\[
  \frac2m\binom k2\leq(k-1)_+\leq\binom k2.
\]
Applying this pointwise to $N_t$ gives
\begin{equation}\label{eq:overlap-pair-bounds}
  \frac2m\sum_{i<j}\Pp(A_i(t)\cap A_j(t))
  \leq D_t
  \leq\sum_{i<j}\Pp(A_i(t)\cap A_j(t)).
\end{equation}
Let $\rho_{ij}=S_{ij}$.  If some $\rho_{ij}>0$, the left inequality in
\eqref{eq:overlap-pair-bounds} and \Cref{lem:bivariate-upper-tails} imply
$D_t/p_t^2\to\infty$, contradicting
\eqref{eq:overlap-deficit-limit}.  Hence every $\rho_{ij}\leq0$.

For nonpositive correlations, \Cref{lem:bivariate-upper-tails} and the right
inequality in \eqref{eq:overlap-pair-bounds} give
\[
  \binom m2
  =\lim_{t\to\infty}\frac{D_t}{p_t^2}
  \leq\#\{(i,j):i<j,\ \rho_{ij}=0\}
  \leq\binom m2.
\]
Therefore every off-diagonal covariance is zero, and $S=I_m$.  The converse
is immediate from independence.
\end{proof}

\section{Relation to two prior announced proofs}
\label{app:prior-proofs}

Goldsmith \cite{Goldsmith2021} and Pastore \cite{Pastore2023} announce proofs
of the Simplex Mean Width and Weak Simplex conjectures, respectively.  We
record why we do not use the current posted versions as prior proofs.  The
discussion concerns only arXiv:2112.03393v2, revised 27 June 2023, and
arXiv:2306.13478v2, revised 13 November 2023.  All page, section, theorem, and
equation references below refer to those versions.  It is not a claim that
either conjecture is false, that either strategy cannot be repaired, or that
the authors have no later argument.

\subsection{Goldsmith's centroid-monotonicity step}

The dependency in \cite{Goldsmith2021} is
\[
  \text{Theorem 3.1}
  \Longrightarrow
  \text{Corollary 3.2}
  \Longrightarrow
  \text{Theorem 4.1}.
\]
Theorem~3.1 is a strict monotonicity statement for spherical centroids under a
longitudinal shift.  On pp.~7--9, its proof uses local approximations and
reduces the pointwise Pr\'ekopa--Leindler hypothesis to a sufficient
inequality.  At the end of p.~9, the manuscript observes that this inequality
fails as the dimension tends to infinity and derives only a restricted
angular condition.  The proof then ends without another estimate verifying
the required pointwise hypothesis.  Corollary~3.2 invokes Theorem~3.1, and
Theorem~4.1 invokes Corollary~3.2.  Thus the posted argument does not establish
the claimed all-dimensional theorem.

\subsection{Pastore's centroid-injectivity step}

The injectivity argument in Section~III-D of \cite{Pastore2023} defines
conditional means $t$ and $t'$ over the differences
$\widetilde D\setminus\widetilde D'$ and
$\widetilde D'\setminus\widetilde D$ in (18c)--(18d).  The appendix obtains
the set inclusions (23b)--(23c) and then treats the means as points of the
corresponding differences when deriving the sign statements (24b)--(24c).
A conditional mean belongs to the closed convex hull of its support, not
necessarily to the support itself, and a difference of convex cones need not
be convex.  The failed inequality in (24b) or (24c) may depend on the sampled
point and need not survive averaging.

The issue occurs within the Gaussian simplicial-cone setting.  In $\R^3$ let
\[
 D=\operatorname{cone}(e_1,e_2,e_3),\qquad
 w_a=\frac{(1,a,a)}{\sqrt{1+2a^2}},\qquad
 D'_a=\operatorname{cone}(w_a,e_2,e_3),
\]
where $a>0$.  Then
\[
 D'_a=\{x_1>0,\ x_2>a x_1,\ x_3>a x_1\}\subset D.
\]
Let $Z\sim\cN(0,I_3)$, put $A_a=D\setminus D'_a$, and write
$Q=1-\Phi$.  Define
\begin{align}
 I_a
 &:=\int_0^\infty\varphi(x)\varphi(ax)Q(ax)\,dx \notag\\
 &=\frac{\varphi(0)}{\sqrt{1+a^2}}
   \left[
     \frac14-\frac1{2\pi}
       \arctan\!\left(\frac{a}{\sqrt{1+a^2}}\right)
   \right].                                                   \label{eq:Pastore-Ia}
\end{align}
For fixed $Z_1=x>0$, the event $A_a$ has conditional probability
$1/4-Q(ax)^2$.  Integration by parts therefore gives
\begin{align}
 \E[Z_1\ind_{A_a}]
 &=\int_0^\infty x\varphi(x)\bigl(1/4-Q(ax)^2\bigr)\,dx
   =2aI_a.                                                     \label{eq:Pastore-moment1}
\end{align}
Similarly, subtracting the contribution of
$\{Z_2>ax,Z_3>ax\}$ from that of $\{Z_2>0,Z_3>0\}$ gives
\begin{align}
 \E[Z_2\ind_{A_a}]=\E[Z_3\ind_{A_a}]
 &=\int_0^\infty\varphi(x)
    \left(\frac{\varphi(0)}2-\varphi(ax)Q(ax)\right)dx \notag\\
 &=\frac{\varphi(0)}4-I_a.                                      \label{eq:Pastore-moment23}
\end{align}
For $a=1/4$, equations \eqref{eq:Pastore-Ia}--\eqref{eq:Pastore-moment23}
yield
\[
 \frac{\E[Z_2\ind_{A_a}]}{\E[Z_1\ind_{A_a}]}
 =\frac{\E[Z_3\ind_{A_a}]}{\E[Z_1\ind_{A_a}]}
 \approx0.42958>\frac14.
\]
Normalization by $\Pp(A_a)$ does not change these ratios.  Hence
\[
  \E[Z\mid Z\in D\setminus D'_{1/4}]
  \in\operatorname{int}D'_{1/4},
\]
not in $D\setminus D'_{1/4}$.

This example can also be placed under the centroid hypothesis used in (18c).
For the positive orthant $D$, let $G_c\sim\cN(c,\sigma^2I_3)$ with
$c\in D\cap S^2$.  Then
\[
  \Pp\{G_c\in D\}=\prod_{j=1}^3\Phi(c_j/\sigma).
\]
As a function of $q_j=c_j^2$, the logarithm of each factor is strictly
concave: its derivative is proportional to
$r(\sqrt{q_j}/\sigma)/\sqrt{q_j}$, which is strictly decreasing.  Symmetry and
strict concavity therefore give the unique centroid
$c_D=(1,1,1)/\sqrt3$.  If $X_\sigma=c_D+\sigma Z$, conic homogeneity and
dominated convergence imply
\[
 \frac1\sigma\E[X_\sigma\mid X_\sigma\in A_{1/4}]
 \longrightarrow
 \E[Z\mid Z\in A_{1/4}]
 \in\operatorname{int}D'_{1/4}.
\]
Thus, for all sufficiently large $\sigma$, the conditional mean in (18c),
with the relevant Gaussian centroid, lies in $D'_{1/4}$ rather than in
$D\setminus D'_{1/4}$.  The sign conclusions (24b)--(24c), used in
(19)--(20) to prove injectivity, therefore do not follow.

There is a second, independent issue in the same appendix.  It asserts that
the oriented facet normals $e_1,e_i,e_i'$ are pairwise linearly independent,
so every coefficient in
$\lambda_1e_1+\lambda_ie_i+\lambda_i'e_i'=0$ is nonzero and one may normalize
$\lambda_1=1$ before defining the partition in (22).  Let
\[
 \widetilde V=(u_1,u_2,u_3),\qquad
 \widetilde V'=\left(\frac{u_1+u_3}{\sqrt2},u_2,u_3\right),
\]
where $u_1,u_2,u_3$ are the standard basis.  Both triples are linearly
independent and differ only in their first generator.  For $i=2$,
\[
 \operatorname{span}(u_1,u_3)
 =\operatorname{span}\!\left(\frac{u_1+u_3}{\sqrt2},u_3\right),
\]
so the corresponding oriented normals satisfy $e_2=e_2'=u_2$.  The valid
dependence $e_2-e_2'=0$ has $\lambda_1=0$, and the normalization and case
partition in (22) do not cover this admissible configuration.

Both issues occur in the injectivity result used in Section~III-E to deduce
reflection symmetry.  They are independent of the present proof.  We
therefore cite \cite{Goldsmith2021,Pastore2023} as prior announced proofs but
do not use their asserted conclusions.

\section{Proof of the finite-energy AWGN corollary}
\label{app:awgn}

We prove \Cref{cor:finite-energy}.  The channel is the real AWGN sequence
model
\begin{equation}\label{eq:awgn-model}
  Y=X+Z,
  \qquad Z_j\stackrel{\mathrm{iid}}{\sim}\cN(0,N_0/2).
\end{equation}
The encoder is deterministic, the $M$ messages are equiprobable, there is no
feedback, and every codeword $c_i\in\ell_2$ satisfies
$\norm{c_i}^2\leq E$.  The decoder is arbitrary, and there is no constraint on
the number of channel uses or nonzero codeword coordinates.  This is the
finite-energy model of \cite[Definition~1]{PolyanskiyPoorVerdu2011}.

The output sequence $Z$ is almost surely not in $\ell_2$, so the observation
cannot be projected directly onto a finite-dimensional subspace of $\ell_2$.
Nevertheless, the likelihood ratios depend on finitely many Gaussian linear
statistics.  We construct those statistics first.

\subsection{Finite-dimensional sufficient statistics}

Work on the canonical output space
\[
  \Omega=\R^{\mathbb N},
  \qquad
  \mathcal A=\mathcal B(\R)^{\otimes\mathbb N},
\]
with coordinate maps $Y_k(y)=y_k$.  Let $P_i$ be the product Gaussian law on
$(\Omega,\mathcal A)$ when $c_i$ is transmitted, and set
\[
  D=\operatorname{span}\{c_i-c_1:1\leq i\leq M\}\subset\ell_2.
\]
Choose an orthonormal basis $u_1,\ldots,u_r$ of $D$, where $r\leq M-1$, and
define the measurable partial sums
\[
  S_{j,n}(y)=\sum_{k=1}^n u_{j,k}y_k.
\]
We use one measurable version of the limit for all hypotheses:
\begin{equation}\label{eq:sufficient-statistic}
  T_j(y)=
  \begin{cases}
    \displaystyle\lim_{n\to\infty}S_{j,n}(y),
      &\text{if the limit exists},\\[1ex]
    0,&\text{otherwise},
  \end{cases}
  \qquad 1\leq j\leq r.
\end{equation}
The convergence set is Borel and the limit there is a pointwise limit of
measurable functions, so $T_j$ is $\mathcal A$-measurable.

Under message $i$, write $Y_k=c_{i,k}+Z_k$.  The deterministic series
$\sum_k u_{j,k}c_{i,k}$ converges by Cauchy--Schwarz.  The random series
$\sum_k u_{j,k}Z_k$ consists of independent centered Gaussian variables and
\[
  \sum_{k=1}^{\infty}\operatorname{Var}(u_{j,k}Z_k)
  =\frac{N_0}{2}\sum_{k=1}^{\infty}u_{j,k}^2
  =\frac{N_0}{2}.
\]
It therefore converges almost surely and in $L^2(P_i)$.  Thus
\eqref{eq:sufficient-statistic} is a common version under every $P_i$, and
\begin{equation}\label{eq:T-law}
  T(Y)\sim\cN\!\left(
      (\ip{c_i}{u_j})_{j=1}^r,\frac{N_0}{2}I_r
    \right)
  \qquad\text{under }P_i.
\end{equation}

For $d_i=c_i-c_1$, write $d_i=\sum_{j=1}^r\beta_{ij}u_j$ and define
\[
  \ip{Y}{d_i}:=\sum_{j=1}^r\beta_{ij}T_j(Y).
\]
Under every hypothesis this agrees almost surely with
$\lim_n\sum_{k=1}^n d_{i,k}Y_k$.  The Cameron--Martin formula
\cite{Bogachev1998} gives
\begin{equation}\label{eq:infinite-likelihood-ratio}
  L_i(Y):=\frac{dP_i}{dP_1}(Y)
  =\ell_i(T(Y)),
\end{equation}
where
\begin{equation}\label{eq:finite-likelihood-function}
  \ell_i(t)
  =\exp\!\left\{
      \frac{2}{N_0}\sum_{j=1}^r\beta_{ij}t_j
      -\frac1{N_0}\bigl(\norm{c_i}^2-\norm{c_1}^2\bigr)
    \right\}.
\end{equation}
To justify \eqref{eq:infinite-likelihood-ratio}, let $L_{i,n}$ be the
likelihood ratio based on the first $n$ coordinates.  Then $(L_{i,n})_n$ is a
nonnegative $P_1$-martingale and
\[
  \sup_n\E_{P_1}L_{i,n}^2
  \leq \exp\!\left(\frac{2\norm{d_i}^2}{N_0}\right)<\infty.
\]
Hence $L_{i,n}$ converges in $L^1(P_1)$ to the expression in
\eqref{eq:infinite-likelihood-ratio}.  In particular, every likelihood ratio
is measurable with respect to $T=(T_1,\ldots,T_r)$.

Let $\ell_1\equiv1$ and let $\overline P=M^{-1}\sum_{i=1}^M P_i$ be the output
law under the uniform prior.  Bayes' formula gives
\begin{equation}\label{eq:posterior-sufficient}
  \Pp\{I=i\mid Y\}
  =\frac{\ell_i(T(Y))}{\sum_{k=1}^M \ell_k(T(Y))}
  \qquad \overline P\text{-almost surely}.
\end{equation}
Thus a Bayes-optimal decoder, equivalently a maximum-likelihood decoder under
the uniform prior, is a function of $T$.  Conversely, every decoder based on
$T$ is also available in the full experiment.  The sequence model and the
finite-dimensional Gaussian model in \eqref{eq:T-law} therefore have the same
optimal probability of correct decoding.

\subsection{Optimality of a full-energy regular simplex}

\begin{proposition}\label[proposition]{prop:awgn-reduction}
Among all $M$-message codes in the preceding model, the optimal average
probability of correct decoding is attained by a centered regular
$(M-1)$-simplex whose vertices have squared norm $E$.
\end{proposition}

\begin{proof}
Start with arbitrary codewords $c_1,\ldots,c_M\in\ell_2$ satisfying
$\norm{c_i}^2\leq E$.  Append one independent signal/noise coordinate; after
relabeling the countable coordinates, this is the isometric embedding
$\ell_2\hookrightarrow\ell_2\oplus\R\cong\ell_2$.  Set
\begin{equation}\label{eq:energy-equalization}
  \widetilde c_i
  =\left(c_i,\sqrt{E-\norm{c_i}^2}\right).
\end{equation}
Every augmented codeword has squared norm $E$.  A decoder in the augmented
experiment may ignore the additional observation coordinate and use the
original decoder.  Thus optimal correct-decoding probability cannot decrease
under this augmentation.

Let
\[
  \widetilde D
  =\operatorname{span}\{\widetilde c_i-\widetilde c_1:1\leq i\leq M\}.
\]
By the sufficient-statistic construction above, the augmented experiment is
equivalent to a Gaussian shift experiment on $\widetilde D$.  All
$\widetilde c_i$ have the same orthogonal projection $w$ onto
$\widetilde D^\perp$.  Write
\begin{equation}\label{eq:common-translation}
  \widetilde c_i=w+v_i,
  \qquad v_i\in\widetilde D.
\end{equation}
The sufficient statistics have means $v_i$; the common component $w$ does not
appear.  Since $w\perp\widetilde D$ and the $\widetilde c_i$ have equal norm,
\begin{equation}\label{eq:common-radius}
  \norm{v_i}^2=E-\norm{w}^2=:r^2
\end{equation}
is independent of $i$, with $0\leq r\leq\sqrt E$, and
$\dim\widetilde D\leq M-1$.

If $r>0$, divide the observation by $\sqrt{N_0/2}$ and apply
\Cref{cor:mgf} together with the maximum-likelihood identity to the unit
vectors $v_i/r$, after identifying $\widetilde D$ with a Euclidean space and
embedding it in $\R^{M-1}$ if necessary.  Their probability of correct
decoding is no larger than that of a regular $(M-1)$-simplex of radius $r$.
If $r=0$, all finite-dimensional means coincide and the optimal probability is
$1/M$.

We now compare different radii.  Let $0\leq r_1\leq r_2$, put
$\sigma^2=N_0/2$, and consider a fixed regular simplex
$x_1,\ldots,x_M$.  If $r_2>0$, set $\eta=r_1/r_2$.  From the observation
$Y_2=r_2x_I+\sigma Z$, form
\begin{equation}\label{eq:radius-transformation}
  K(Y_2)=\eta Y_2+\sigma\sqrt{1-\eta^2}\,Z',
\end{equation}
where $Z'$ is an independent standard Gaussian vector.  Conditional on $I$,
$K(Y_2)$ has the same law as $r_1x_I+\sigma Z''$.  Thus the radius-$r_1$
experiment is a message-independent randomized degradation of the
radius-$r_2$ experiment.  Composing a decoder for the former with this kernel
may produce a randomized decoder for the latter.  Randomization cannot
improve Bayes success for finitely many hypotheses under
zero--one loss: conditional on an observation, success is a linear function
of the decision probabilities and is maximized by a deterministic
maximum-a-posteriori rule.  The optimal probability is therefore
nondecreasing in $r$.  Equivalently, the
smaller-radius experiment is a Blackwell degradation of the larger-radius
experiment \cite{Blackwell1953}.  A regular simplex of radius
$\sqrt E$ is feasible and attains the resulting upper bound.
\end{proof}

The equalization step uses the maximal energy constraint on each codeword.  It
does not yield the same conclusion under an average-over-codebook energy
constraint.

\subsection{Evaluation of the simplex decoding probability}

Let $e_1,\ldots,e_M$ be the standard basis of $\R^M$ and set
\begin{equation}\label{eq:orthogonal-amplitude}
  A=\sqrt{\frac{EM}{M-1}},
  \qquad
  u_i=Ae_i,
  \qquad
  \bar u=\frac1M\sum_{j=1}^M u_j.
\end{equation}
The centered vectors $u_i-\bar u$ form a feasible regular simplex and
\begin{equation}\label{eq:centered-orthogonal-radius}
  \norm{u_i-\bar u}^2=A^2\left(1-\frac1M\right)=E.
\end{equation}
The orthogonal representatives $u_i$ have squared norm $EM/(M-1)>E$ when
$M>1$ and are not asserted to be feasible codewords.  They are used only to
evaluate the equivalent translated experiment: adding the common vector
$\bar u$ to every signal and to the observation is a deterministic bijection
and leaves all likelihood ratios and optimal decoding probabilities
unchanged.

Suppose message $1$ is sent and write $\sigma=\sqrt{N_0/2}$.  In the
orthogonal representation, correct maximum-likelihood decoding occurs when
\[
  A+\sigma W_1\geq\sigma W_j\qquad(2\leq j\leq M),
\]
for independent standard Gaussian random variables
$W_1,\ldots,W_M$.  Conditioning on $W_1=W$ gives
\begin{equation}\label{eq:conditional-pc}
  P_{\mathrm c}^*(E,M)
  =\E\Phi\!\left(W+\frac A\sigma\right)^{M-1}
  =p_M\!\left(\sqrt{\frac{2EM}{(M-1)N_0}}\right).
\end{equation}
By symmetry this is also the average success probability, proving
\eqref{eq:pc-opt}.

The function $p_M$ from \eqref{eq:pm-def} is continuous and
\begin{equation}\label{eq:pm-derivative}
  p_M'(a)=(M-1)\E\!\left[
       \Phi(W+a)^{M-2}\varphi(W+a)
     \right]>0.
\end{equation}
At $a=0$, the random variable $\Phi(W)$ is uniform on $(0,1)$, so
$p_M(0)=1/M$.  Dominated convergence gives $p_M(a)\uparrow1$ as
$a\to\infty$.  These facts prove the assertions about $a_{M,\epsilon}$ and
$E^*(M,\epsilon)$ in \Cref{cor:finite-energy}.  They also show that the set in
\eqref{eq:M-star} has a largest element.  For fixed $E$, the shift
$\sqrt{2EM/((M-1)N_0)}$ remains bounded as $M\to\infty$, whereas the maximum
of $M-1$ independent standard Gaussian random variables diverges in
probability.  Therefore the correct-decoding probability tends to zero.

\begingroup
\small
\bibliographystyle{alpha}
\bibliography{references}
\endgroup

\end{document}